\documentclass[a4paper,10pt,leqno,english]{smfart}
\usepackage{aeguill}
\usepackage{enumerate}
\usepackage{amssymb,amsmath,latexsym,amsthm}
\usepackage[T1]{fontenc}
\usepackage{smfthm}      
\usepackage{geometry}   
\usepackage{url}      
\usepackage[frenchb, english]{babel}
\usepackage[utf8]{inputenc}   
\usepackage{mathrsfs}
\usepackage{xcolor}
\usepackage{comment}
\usepackage{dsfont}
\usepackage{yhmath}
\definecolor{violet}{rgb}{0.0,0.2,0.7}
\definecolor{rouge}{cmyk}{0.0,0.6,0.4,0.3}
\definecolor{rouge2}{rgb}{0.8,0.0,0.2}
\usepackage{hyperref}
\usepackage{hyperref}

\hypersetup{
    bookmarks=true,         
    unicode=false,          
    pdftoolbar=true,        
    pdfmenubar=true,        
    pdffitwindow=false,     
    pdfstartview={FitH},    
    pdftitle={},    
    pdfauthor={},     
    colorlinks=true,       
   linkcolor=violet,          
    citecolor=black,        
    filecolor=black,      
    urlcolor=cyan}           
\newcommand{\R}{\mathbb{R}}
\newcommand{\CC}{\mathbb{C}}

\renewcommand{\d}{\partial}

\newcommand{\vp}{\varphi}

\renewcommand{\O}{\mathcal{O}}
\newcommand{\Ox}{\mathcal{O}_{X}}

\newcommand{\ep}{\varepsilon}

\newcommand{\la}{\langle}
\newcommand{\ra}{\rangle}

\renewcommand{\ge}{\geqslant}
\renewcommand{\le}{\leqslant}

\newcommand{\om}{\omega}

\newcommand{\D}{D}

\newcommand{\iddb}{dd^c}
\newcommand{\ddc}{dd^c}
\newcommand{\Fe}{F_{\varepsilon}}
\newcommand{\Supp}{\mathrm{Supp}}

\newcommand{\vpe}{\varphi_{\varepsilon}}
\newcommand{\pse}{\psi_{\varepsilon}}
\newcommand{\fe}{F_{\varepsilon}}

\newcommand{\epp}{(\varepsilon^2+|z^p|^2e^{-\varphi_p})}
\newcommand{\epq}{(\varepsilon^2+|z^q|^2e^{-\varphi_q})}

\newcommand{\fepp}{(\varepsilon^2+|z^p|^2)}
\newcommand{\fepq}{(\varepsilon^2+|z^q|^2)}
\newcommand{\fepr}{(\varepsilon^2+|z^r|^2)}
\newcommand{\feps}{(\varepsilon^2+|z^s|^2)}
\newcommand{\fepu}{(\varepsilon^2+|z^u|^2)}
\newcommand{\fept}{(\varepsilon^2+|z^t|^2)}

\newcommand{\dej}{\tau_j}
\newcommand{\dep}{\tau_p}

\renewcommand{\b}{\bar}
\renewcommand{\om}{\omega}
\newcommand{\omi}{\omega_{\infty}}

\newcommand{\ome}{\omega_{\varepsilon}}

\newcommand{\tr}{\mathrm{tr}}
\newcommand{\Ric}{\mathrm{Ric} \,}

\def\dbar{\overline\partial}
\def\ddbar{\partial\overline\partial}
\def\cO{{\mathcal O}}

\def\cR{{\mathcal R}}
\def\cB{{\mathcal B}}
\def\cC{{\mathscr C}}
\let\ol=\overline

\def\bC{{\mathbb C}}
\def\bR{{\mathbb R}}

\newtheorem*{thma}{Theorem A}
\newtheorem*{thmb}{Theorem B}
\newtheorem*{thmc}{Theorem C}

\begin{document}
\title[Metrics with cone singularities along normal crossing divisors]{Metrics with cone singularities along normal crossing divisors and holomorphic tensor fields}
\date{\today}
\author{Fr\'ed\'eric Campana} 
\address{Institut Elie Cartan \\
Universit\'e Henri Poincar\'e\\
B. P. 70239, F-54506 Vandoeuvre-l\`es-Nancy Cedex, France}
\email{Frederic.Campana@iecn.u-nancy.fr}

\author{Henri Guenancia}
\address{Institut de Mathématiques de Jussieu \\
Université Pierre et Marie Curie \\
 Paris}
\email{guenancia@math.jussieu.fr}
\urladdr{www.math.jussieu.fr/~guenancia}

\author{Mihai P\u aun}
\address{Institut Elie Cartan \\
Universit\'e Henri Poincar\'e\\
B. P. 70239, F-54506 Vandoeuvre-l\`es-Nancy Cedex, France\\
and Korea Institute for Advanced Study, Korea}
\email{Mihai.Paun@iecn.u-nancy.fr}

\thanks{It is a pleasure to thank S. Boucksom and V. Tosatti: their remarks and comments 
helped us to improve the content and the presentation of the present article.}

\begin{abstract} 
We prove the existence of non-positively curved Kähler-Einstein metrics with cone singularities along a given simple normal crossing divisor of a compact Kähler manifold, under a technical condition on the cone angles, and we also discuss the case of positively-curved Kähler-Einstein metrics with cone singularities. As an application we extend to this setting classical results of Lichnerowicz and Kobayashi on the parallelism and vanishing of appropriate holomorphic tensor fields.
\end{abstract}

\begin{altabstract} 
Dans cet article, nous prouvons l'existence de métriques de Kähler-Einstein à courbure négative ayant des singularités coniques le long d'un diviseur à croisements normaux simples sur une variété kählerienne compacte, sous une hypothèse technique sur les angles des cones. Nous discutons également du cas des métriques de Kähler-Einstein à courbure strictement positive avec des singularités coniques. Nous en déduisons que les résultats classiques de Lichnerowicz et Kobayashi sur le parallélisme et l'annulation des champs de tenseurs holomorphes s'étendent à  notre cadre. 
\end{altabstract}

\maketitle

\tableofcontents

\section{Introduction}

Let $X$ be a $n$-dimensional compact Kähler manifold, and let $D= \sum a_i D_i$ be an effective $\R$-divisor with strictly normal crossing support, such that for all $i$, $0<a_i<1$. In the terminology of the Minimal Model Program, the pair $(X, D)$ is called a \textit{log-smooth klt pair}; following \cite{Camp2}, \cite{Camp1}, we may also call it a \emph{smooth geometric orbifold}. \\

One may define for such a pair the notion of \textit{cone metric}, or also \textit{metric with cone singularities along $D$}: it corresponds to an equivalence class (up to quasi-isometry) of (Kähler) metrics $\omega_{\rm cone}$ on $X_0= X \setminus \mathrm{Supp}(D)$
having the following property: there exists $C>0$ such that for every point $p \in \mathrm{Supp}(D)$ such that in  $\mathrm{Supp}(D) \, \cap \, \Omega = (z^1\cdots z^d =0)$, where we denote by $\Omega$  
a coordinate neighborhood of $p$, we have: 
\[C^{-1} \omega_{\textrm{o}} \le \omega_{\rm cone} \le C \,\omega_{\textrm{o}}\]
where 
\[ \omega_{\textrm{o}} = \sqrt{-1}\left(\sum_{i=1}^d   \frac{dz^i\wedge d\bar z^{i}}{|z^i|^{2a_i}} +\sum_{i=d+1}^n dz^i\wedge d\bar z^{i}\right)\]
is the standard cone metric on $\CC^n\setminus \mathrm{Supp}(D)$ with respect to the divisor 
\[D=\sum_{i=1}^d a_i [z^i=0].\]

In the context of geometric orbifolds the notion of \emph{holomorphic tensors} was first formulated in \cite{Camp2}, \cite{Camp1}; as we will see here, it corresponds to holomorphic tensors on $X_0= X \setminus \mathrm{Supp}(D)$ (in the usual sense) which are bounded with respect to one cone metric along $D$. Roughly speaking, the results of this paper consist in constructing cone metrics with prescribed Ricci curvature by means of Monge-Ampère equations, and use them to show the vanishing or parallelism of some holomorphic tensors on $(X,D)$ according to the sign of the adjoint bundle $K_X+D$.\\\\

We offer next a more extensive presentation of the main theorems obtained in this article. In the following, we fix a log-smooth klt pair $(X,D)$ where we choose to write $D$ in the form
\[D:= \sum_{j\in J}(1-\tau_j)Y_j\]
for some smooth hypersurfaces $Y_j$ having strictly normal intersections, and some real numbers $0< \tau_j < 1$. The numbers $\tau_j$ (or more precisely $2\pi \tau_j$) have a geometric interpretation in terms of the cone angles.
As we said, our goal is to construct a K\"ahler metric $\omega_{\infty}$ on $X_0:=X\setminus (\cup Y_j)$ whose Ricci curvature is 
given according to the sign of the $\bR$-divisor $K_X+ D$, 
with prescribed asymptotic along $D$. For example, let us consider the case where $K_X+D$ is ample. Then we want $\omega_{\infty}$ to satisfy the two following properties: \\
\begin{enumerate}
\item[$\bullet$] $\omega_{\infty}$ is Kähler-Einstein on $X_0$: $-\Ric(\omega_{\infty})= \omega_{\infty} \, \,  \textrm{on} \,\, X_0$
\item[$\bullet$] $\omega_{\infty}$ has cone singularities along $D$.\\
\end{enumerate}

This problem, in a more general form not involving positivity on the adjoint bundle $K_X+D$, can be restated in the Monge-Ampère setting. Namely, given a Kähler form $\omega$ on $X$, we want to solve the following equation (in $\vp$):
\[(\om+ \iddb \vp)^n = e^{f+\lambda \vp} \mu_D \leqno{(\mathrm{MA})}\]
for $\lambda \in \{0, 1\}$, $f\in \cC^{\infty}(X)$, and where $\mu_D$ is the volume form on $X$ given by
\[\mu_D = \frac{\om^n}{\prod_{j\in J} |s_j|^{2(1-\tau_j)}}\]
for sections $s_j$ of $(L_j, h_j)$ defining $Y_j$; moreover, if $\lambda=0$, one assumes that $\int_X e^f \mu_D=\int_X \om^n$. Here $d= \partial+ \overline{\partial}$ and $d^c= \frac{1}{2i\pi}(\partial - \overline{\partial})$.\\

When $\lambda=1$, using an elementary regularization argument, one can construct a (unique) continuous solution $\vp$ of equation (MA), cf. section \ref{subsec:est0}. The case where $\lambda=0$ is deeper: as $\mu_D$ as $L^p$ density for some $p>1$ (this is an important place where the klt condition is used), a theorem of S. Ko\l odziej \cite{Kolo} shows that this equation has a unique (normalized) solution $\vp_{\infty}$, which is continuous on $X$. By \cite[Theorem B]{EGZ} (see also \cite{Paun}), $\vp_{\infty}$ is known to be smooth outside $\mathrm{Supp}(D)$.\\

The metric $\om_{\infty}=\om + \iddb \vp_{\infty}$ on $X_0=X\setminus(\cup Y_j)$ satisfies the condition stated in the first item. 
Unfortunately, the theorem of S. Ko\l odziej or their generalizations don't give us order $2$ information on $\vp_{\infty}$ near $\mathrm{Supp}(D)$, which is exactly what the condition in the second item ("cone singularities") requires. \\

\noindent Our first main result is the following:  

\begin{thma}[Main Theorem]
We assume that the coefficients of $D$ satisfy the inequalities
\[ 0< \tau_j\le \frac{1}{2}.\] Then the Kähler metric $\omega_{\infty}$ solution to the following equation 
\[(\om+ \iddb \vp)^n =  \frac{e^{f+\lambda \vp}}{\prod_{j\in J} |s_j|^{2(1-\tau_j)}} \, \om^n \leqno{(\mathrm{MA})}\]
 has cone singularities along $D$.
 \end{thma}
\medskip
 
 This problem has already been studied in many important particular cases.
Indeed, in the standard orbifold case, corresponding to coefficients $\displaystyle \tau_j=\frac 1 {m_j}$ for some natural numbers $m_j \ge 2$, G. Tian and S.-T. Yau have established in \cite{Tia} the existence of such a Kähler-Einstein metric compatible with the (standard) orbifold structure, in the case where $K_X+D$ is ample. R. Mazzeo (\cite{Maz}) annonced the existence of Kähler-Einstein metrics with cone singularities when $D$ is smooth and irreducible (assuming $K_X+D$ ample) while T. Jeffres (\cite{Jef}) studied the uniqueness under the same assumptions. Moreover the recent article by S. Donaldson 
(cf. \cite{Don} and the references therein) is very much connected with the result above; even more recently, we refer to the papers by S. Brendle  (cf. \cite{Brendle}) and R. Berman (cf. \cite{rber}) 
for the complete analysis of the Ricci-flat (resp. positive Ricci curvature) case under the assumption that $D$ had only one smooth component. After the first version of this article appeared, T. Jeffres, R. Mazzeo and Y. Rubinstein gave in \cite{JMR} a complete treatment of the Kähler-Einstein problem for metrics with cone singularities along one smooth divisor. \\ 

We note that the assumption $\displaystyle \tau_j\le {1\over 2}$ is automatically satisfied in the orbifold case, and that it is also appears in a crucial way in \cite{Brendle} so as to bound the holomorphic bisectional curvature of the cone metric outside the aforesaid hypersurface.\\

We discuss now briefly our approach to the proof of this result. The strategy is to regularize the equation $(\mathrm{MA})$ and to obtain uniform estimates; then $\omi$ will be obtain as a limit point of solutions of the regularized equations.  In order to achieve this goal, we will proceed as follows.
We first approximate the standard cone metric $\omega_{\rm o}$ 
(or better say, its global version) with a sequence of smooth K\"ahler metrics $\displaystyle(\ome)_{\varepsilon> 0}$
on $X$. \\

The approximations are constructed such that 
$\omega_{\varepsilon}:= \om+ \iddb \pse$ belongs to a fixed cohomology class $[\om]$ for some metric $\om$ on $X$. The explicit expression of $\psi_\ep$ 
is given in section \ref{sec2}. 
Our candidate for the sequence 
converging to the metric $\omega_{\infty}$ we seek 
will be
\[\om_{\vpe}:=\ome+ \iddb \vpe\] where $\vpe$ is the solution 
to the following Monge-Amp\`ere equation, which may be seen as a regularization of equation $(\textrm{MA})$:
\[\om_{\vpe}^n=\frac{e^{f+\lambda(\pse+\vpe)}}{\prod_{j=1}^d(\ep^2+|s_j|^2)^{1-\tau_j}}\om^n.\leqno (\star_{\ep})\]
Here we have $\lambda \in \{0, 1\}$, $f\in \mathscr C^{\infty}(X)$
(this function will be given by the geometric context in the second part 
of this article), and the $s_j$'s are sections of hermitian line bundles $(L_j, h_j)$ such that $Y_j=(s_j=0)$. If $\lambda=0$, we impose moreover the normalization
\[\int_X \vpe \, dV_{\omega}=0.\]

First we remark that if
$\varepsilon> 0$, then we can solve the equation $(\star_{\ep})$ and obtain a solution $\varphi_\varepsilon\in \cC^\infty(X)$ thanks to the fundamental theorem of Yau in \cite{Yau78}. Indeed, this can be done since $\displaystyle \omega_{\varepsilon}$ is a genuine metric, i.e. it is smooth. Of course, the main part of our work is to analyze the uniformity properties of the family of functions
$$(\varphi_\varepsilon)_{\varepsilon> 0}\leqno(39)$$ 
as $\varepsilon\to 0$. To this end,
we mimic the steps of the "closedness" part of the method of continuity in \cite{Yau78}, as follows.\\

\begin{itemize}
\item Using the results of \cite{Kolo}, we already obtain $\cC^0$ estimates; this combined with standard results in the theory of Monge-Amp\`ere equations gives us \textit{interior} $\cC^{2, \alpha}$ estimates \emph{provided that} global  
$\cC^{2}$ estimates have been already established. If we fulfill this program, than we can extract from $(\om_{\vpe})_{\ep>0}$ a subsequence converging to the desired solution $\omi$, which will henceforth be smooth outside the support of $D$. \\

\item As we mentioned earlier, we aim to compare $\omi$ and $\om_o$, and to this end we need \textit{global} $\cC^2$ estimates on $\vp_{\ep}$, i.e. to compare $\displaystyle \omega_{\varphi_\ep}$ and $\ome$ in an uniform manner. The key observation, though rather simple, is that in our situation, we only need to obtain a uniform lower bound on the holomorphic curvature of $\om_{\ep}$ so as to get the estimates. In the next section, we detail the preceding observation, providing a general context under which one may obtain such $\cC^2$ estimates. \\

\end{itemize}

In conclusion, we do not have to deal directly with the 
singular metric $\omega_o$ since we are using its regularization family
$\displaystyle (\omega_{\varepsilon})_{\varepsilon> 0}$: it is thanks to this simple approach 
that we can avoid the "openness'' part of the continuity
method, and it enables us to 
treat the case where $D$ is not necessarily irreducible, compare with 
 \cite{Don}, \cite{Brendle}, \cite{rber} and the references therein
 (especially \cite{KobR}, \cite{BanKob}).\\\\
 
 We study then equation $(\mathrm{MA})$ in the case where $\lambda=-1$. We know that this equation (even when the volume form is smooth) does not necessarily admit a solution, so that we can't use the same techniques as previously. Therefore, we will only consider the cases where we already know that $(\mathrm{MA})$ admits a solution. More generally, we prove the following result:
 
 \begin{thmb}
 Let $X$ be a compact Kähler manifold and $D=\sum (1-\tau_j) Y_j$ a divisor with simple normal crossings such that its coefficients satisfy $0<\tau_j\le 1/2$. Let $\mu_{D}=dV/\prod_j |s_j|^{2(1-\tau_j)}$ be a volume form with cone singularities along $D$, $\psi$ a bounded quasi-psh function, and $\omega$ a Kähler form on $X$. Then any (bounded) solution $\vp$ of
\[(\om+ \iddb \vp)^n = e^{-\psi} \mu_{D}\]
is Hölder-continuous and the metric $\om+ \iddb \vp$ has cone singularities along $D$.
 \end{thmb}

We may notice that we only need to assume the existence of a solution $\vp$ belonging to the space $\mathcal E(X,\om)\supset PSH(X,\om) \cap L^{\infty}(X)$; we refer to section \ref{logfano} for the details. \\
Now, to get back to our initial equation $(\mathrm{MA})$ with $\lambda=-1$, we apply Theorem \textbf{B} to $\psi=\vp$. This shows in particular that any Kähler-Einstein metric (in a sense to be defined) on a log-Fano manifold $(X,D)$ has cone singularities along $D$, as soon as the coefficients of $D$ are greater than $1/2$.\\
In particular, using the recent results of Berman \cite[Theorem 1.5]{rber}, we see that a Fano manifold $X$ carrying a smooth anticanonical divisor $D\in |-K_X|$ admits a conic Kähler-Einstein metric $\om_{\tau}$ for any $\tau$ sufficiently small - in the sense that \[\Ric \om_{\tau}=\tau \om_{\tau}+(1-\tau) [D]\]
This enables to start the continuity method in the program proposed by Donaldson \cite{Don}.  \\\\

In the second part of this article we will establish a few
results concerning the holomorphic tensors on geometric orbifolds; our theorem 
gives a  
positive answer to a question raised by the first named author in \cite{Camp2},
where these objects were introduced. 
The (orbifold) tensor bundles which are 
$r$-contra\-variant and $s$-covariant are denoted by $T^r_s(X|D)$. They are the locally free $\Ox$-modules whose (local) sections are the usual holomorphic tensors (outside $\mathrm{Supp}(D)$) which are bounded with respect to any metric having cone singularities along $D$. We will explain this definition and give some of the basic properties of $T^r_s(X|D)$ in section \ref{sec3}.\\

We have the following result.

\begin{thmc}

Let $(X, D)$ be a log-smooth klt pair; in addition, we assume that for all $j\in J$, we have
$\displaystyle 0< \tau_j\le 1/2$. Then the following assertions hold true.

\begin{enumerate}
\item[$(i)$] If the Chern class $c_1(K_X+ D)$ contains a K\"ahler metric, then we have
$$H^0\big(X, T_s^r(X|D)\big)= 0,$$
for any $r\ge s+ 1$.
\item[$(ii)$] If $-c_1(K_X+ D)$ contains a K\"ahler metric, then we have
$$H^0\big(X, T_s(X|D)\big)= 0$$
for any $s\ge 1$.
\item[$(iii)$] If $c_1(K_X+ D)$ contains a smooth, semi-positive 
(resp. semi-negative) representative, then the holomorphic sections of the bundle 
$T^r(X|D)$ (resp. $T_s(X|D)$) are parallel.\\
\end{enumerate}
\end{thmc}

Thus Theorem \textbf{C} is a generalization of the classical 
results by S.~Kobayashi 
and A. Lichnerowicz (cf. \cite{Kob80}, \cite{Kob81}, \cite{Li67}, \cite{Li71}) in the setting of pairs.
The meaning of the word "parallel" in the part $(iii)$ of the theorem
above
will become clear in the section \ref{sec4}; we just mention here that this implies the fact that 
the sections of the bundle $T^r(X|D)$ vanishing a point of $X_0$
are identically zero, if $c_1(K_X+ D)$ contains a smooth, semi-positive representative.\\

We explain next the outline of our proof.
Basically, we follow the differential-geometric arguments in the articles 
quoted above (in the same spirit 
as in \cite{DPS}, \cite[Theorem 14.1]{Dem95}): the fundamental difference is that the metrics we are dealing with here are singular, and 
as usual this creates important difficulties. \\

We discuss next Theorem \textbf{C} e.g. the case 
$K_X+ D> 0$; as we mentioned before, under this hypothesis we are able to construct a metric
 $\omega_{\infty}$ which is smooth on $X_0:= X\setminus \cup_{j\in J}Y_j$
 such that
 \begin{equation}
 \label{1}
\Ric(\omega_{\infty})= -\omega_{\infty}
\end{equation}
on $X_0$ and whose eigenvalues are comparable with the ones of the "standard" cone metric associated to $(X, D)$ (see section \ref{sec2}). \\

An $r$-contravariant and 
$s$-covariant tensor $u$ which is only defined on 
$X_0$ will be called \emph{bounded} if the supremum of its norm with respect to
$\omega_\infty$ on $X_0$ is bounded.
We consider the (musical) contraction operator
$$\#: \cC^\infty_{\cB}\big(X_0, T^r_{s}(X_0)\big)\to \cC^\infty_{\cB}\big(X_0, T^s_{r}(X_0)\big)$$
between the spaces of bounded tensors; a simple verification consists in showing that $\#$ is well-defined. This class of tensors is important for us because of the following reason: a holomorphic tensor $u$ is bounded, and so is $\#u$ (cf. \ref{lem32}). We use this remark in 
the following way.
Let
$$u\in H^0\big(X, T_s(X|D)\big)$$
be a holomorphic tensor field; by the Bochner formula, the difference
\begin{equation}
\label{2}
\Vert \dbar (\theta_\varepsilon u)\Vert^2- \Vert \dbar 
\big(\theta_\varepsilon \#(u)\big)\Vert^2
\end{equation}
is given by an operator of order zero, involving the eigenvalues of $\displaystyle \Ric({\omega_{\infty}})$ with respect to $\omega_{\infty}$. Here we denote by 
$\theta_\varepsilon$ a truncation function (which is needed because 
$\displaystyle \omega_{\infty}$ is a genuine metric only on $X_0$);
by equation \eqref{1}, the eigenvalues of $\displaystyle \Ric({\omega_{\infty}})$
are strictly negative, and if we choose the truncation function carefully enough,
we can show that the $L^2$ norm of the \emph{error term}
$$u\otimes \dbar \theta_\varepsilon $$
converges to zero, as $\varepsilon\to 0$. Along the whole procedure
the fact that $\omega_{\infty}$ is comparable with the cone metric is absolutely crucial,
since it enables us to apply the Bochner method with a good control with respect to
$\varepsilon$.


\section{Laplacian estimates revisited}
\label{sec:obs}
\noindent Let $(X, \omega_\varepsilon)$ be a $n$-dimensional compact 
complex manifold, endowed with a sequence of 
K\"ahler metrics. For each $\varepsilon> 0$, we assume that we have
$$\ome= \omega+ dd^c \psi_\varepsilon$$
for some smooth function $\psi_\varepsilon$.
We consider the following equation:
\begin{equation}
\label{eq:ma}
(\omega_\varepsilon+ dd^c \varphi_\varepsilon)^n= e^{F_\varepsilon+\lambda (\vp_\varepsilon+ \psi_\varepsilon)}\omega^n_\varepsilon
\end{equation}
where $\lambda \in \{-1,0,1\}$. If $\lambda= 0$, then we normalize the solution
$\varphi_\varepsilon$ by $\displaystyle \int_X\varphi_\varepsilon \, dV_\omega= 0$. We remark that the equation \eqref{eq:ma} is modeled after 
$(\star_\varepsilon)$, because the factor $$\displaystyle {1\over 
\prod_{j=1}^d(\ep^2+|s_j|^2)^{1-\tau_j}
}$$ will be absorbed by $\ome^n$, by an appropriate choice of 
$\psi_\ep$.\\

Our goal in this section is to derive a geometric context under which
we can obtain a priori \emph{uniform} ${\mathscr C}^2$ estimates for $\varphi_\varepsilon$ in equation \eqref{eq:ma}. This is the content of the following result, crucial in our approach, which is based on two observations: at first, we notice that once we have $\cC^0$ estimates, then one only needs to bound from \textit{below} the holomorphic bisectional curvature of $\ome$ so as to obtain \textit{a priori} $\cC ^2$ estimates (cf. lemma \ref{klem}). The second observation is that one may allow a rhs in the form $e^{\psi_1-\psi_2}$ with $\psi_1,\psi_2$ quasi psh. This borrowed from \cite{BBEGZ}, which itself could be seen as an effective version of results appearing during the proof of the main theorem of \cite{Paun}.

\begin{prop}
\phantomsection
\label{prop}
Let $\omega$ be a Kähler form on $X$, $\psi_1, \psi_2$ smooth functions and $\vp$ a smooth $\om$-psh function such that
\[ (\om + \iddb \vp)^n = e^{\psi_1-\psi_2} \om^n \]
Assume given a constant $C>0$ such that 	
\begin{enumerate}
\item[$(i)$] $\sup_X |\vp| \le C$, 
\item[$(ii)$] $\Delta_{\om} \psi_1 \ge -C $ and $\,  \sup_X |\psi_1| \le C$,
\item[$(iii)$] $dd^c \psi_2 \ge -C \om $ and $\,  \sup_X |\psi_2| \le C$,
\item[$(iv)$] $\Theta_{\om}(T_X) \ge -C \om \otimes \mathrm{Id}_{T_X}$.
\end{enumerate}
Then there exists a constant $A$ depending only on $n$ and $C$ such that 
\[A^{-1} \om \le \om + \iddb \vp  \le A \,\om.\]
\end{prop}

For the computations to follow, it may be useful to translate the (intrinsic) conditions of the previous proposition, especially $(iv)$, in local coordinates. 
Namely, the inequality in $(iv)$ amounts to saying that the following inequality holds:

\begin{equation}
\label{in2}
\sum_{p, q, r, s}R_{p\ol q r \ol s}(z)v^p\ol{v^q}w^r\ol{w^s}\ge - C|v|_{\omega}^2
|w|_{\omega}^2
\end{equation}
for any 
vector fields $\displaystyle v= \sum_p v^p{\partial\over \partial z^p}$ and 
$\displaystyle w= \sum_i w^i{\partial\over \partial z^i}$. 

\noindent The notation in the above relations is as follows:
in local coordinates, we write
$$\omega= \sqrt{-1}\sum_{i, j}g_{ j\ol i} \, dz^j \wedge dz^{\ol i};$$
and the corresponding components of the curvature tensor 
are
$$R_{ j\ol i k \ol l }:= - {\partial ^2{g_{j\ol i}}\over \partial z^k\partial z^{\ol l}}+
\sum_{s, r}g^{s\ol r}{\partial g_{ s\ol i}\over \partial z^{\ol l}}
{\partial g_{ j\ol r}\over \partial z^{k}}.$$\\

The proof of Proposition \ref{prop} is based on the following key lemma: 

\begin{lemm}
\phantomsection
\label{klem}
Let  $\om'=\om+ \iddb  \varphi$ two cohomologous Kähler forms on $X$, and let $f$ be defined by $\om'^{\,n}=e^f \om^n$. Then there exists a constant $B>0$ depending only on a lower bound for the holomorphic bisectional curvature of $\omega$ such that
\[\Delta' \log \tr_{\om}(\om')\ge \, \frac{\Delta f}{\tr_{\om}(\om')}-B \,  \tr_{\om'}(\om)\]
where $\Delta$ (resp. $\Delta'$) is the Laplace operator attached to $\om$ (resp. $\om'$).
\end{lemm}

\begin{proof}
Our starting point is the following result, extracted from \cite[(3.2) p. 99]{Siu}. Gathering terms coming (with different signs) from the scalar and the Ricci curvature together, we will obtain a similar inequality involving only a lower bound for the holomorphic bisectional curvature.


\begin{prop} 
We have the following inequality
\begin{equation}
\label{ineq}
\Delta' (\log\tr_{\om} \om' ) \ge
\frac{1}{\tr_{\om}\om'}\big(-g^{j\ol i}
R_{j\ol i}+ 
 \Delta f+ g^{\prime k\ol l} R^{j\ol i}_{ k\ol l}g^\prime_{j\ol i}\big)
\end{equation}
where we use the Einstein convention.
\end{prop}

\noindent The notations above are as follows:\\
\begin{itemize}
\item In local coordinates, we have
$$\omega= \sqrt{-1}\sum_{i, j}g_{j\ol i} \, dz^j \wedge dz^{\ol i}$$
as well as
$$\om^\prime= 
\sqrt{-1}\sum_{i, j}g^\prime_{ j\ol i} \, dz^j \wedge dz^{\ol i};$$

\item The components of the tensor $\big(R^{j\ol i}_{ k\ol l}\big)$ are obtained from 
$\big(R_{ j\ol i k \ol l}\big)$ by contraction with the metric $\omega$, and 
$$R_{ j\ol i}:= \sum_{p, q}g^{p\ol q} R_{ j\ol i p\ol q}$$
are the coefficients of the Ricci curvature of $\om$ in the $(z)$--coordinates.\\
\end{itemize}

\noindent The right hand side of the inequality \eqref{ineq} is independent of the 
coordinates in which we choose to express the metrics and the corresponding curvature tensors. \\

\noindent Let $p\in X$ be an arbitrary point; we consider a coordinate
system 
$$w= (w^1,\ldots, w^n)$$ on a small open set containing $p$, such that 
$\omega$ is orthonormal
and such that $\om'$ is diagonal at $p$ when expressed in
the $w$-coordinates. 



\noindent We rewrite the rhs term of the inequality \eqref{ineq} in the $w$-coordinates; we have
\begin{equation}
\label{eq:ob1}
\sum_{i, j}g^{j\ol i} R_{ j\ol i}= \sum_{i, l}
R_{ i\ol i l \ol l}(w)
\end{equation}
as well as 
\begin{equation}
\label{eq:ob2}
\sum_{i, j, k, l}g^{\prime k\ol l} R^{ j\ol i}_{k\ol l}
g^\prime_{ j\ol i}= \sum_{i, l}{1+ \varphi_{i \ol i}\over 1+ \varphi_{ l\ol l}}R_{ i\ol i l \ol l}(w)
\end{equation}
The curvature tensor has the following symmetry property
\begin{equation}
\label{eq:ob3}
R_{i\ol i l \ol l}(w)= R_{ l\ol l i \ol i}(w)
\end{equation}
hence by \eqref{eq:ob1}, \eqref{eq:ob2} and \eqref{eq:ob3} combined with the inequality \eqref{ineq} we obtain
\begin{equation}
\label{eq:ob4}
\Delta' (\log\tr_{\om} \om' ) \ge
\frac{1}{\tr_{\om} \om'} \left(\sum_{i\le l} \left({1+ \varphi_{i \ol i}\over 1+ \varphi_{ l\ol l}}+ {1+ \varphi_{l \ol l}\over 1+ \varphi_{ i\ol i}}- 2\right) R_{ i\ol i l \ol l}(w)+ \Delta f \right)
\end{equation}
We choose a lower bound $B>0$ for the holomorphic bisectional curvature:
\begin{equation}
\label{in:minor}
R_{ i\ol i l \ol l}(w)\ge -B,
\end{equation}
so that we obtain successively:
\begin{eqnarray*}
\Delta' (\log\tr_{\om} \om' )& \ge &\frac{1}{\tr_{\om} \om'} \left(\sum_{i\le l} \left({1+ \varphi_{i \ol i}\over 1+ \varphi_{ l\ol l}}+ {1+ \varphi_{l \ol l}\over 1+ \varphi_{ i\ol i}}- 2\right) R_{ i\ol i l \ol l}(w)+ \Delta f \right) \\
&\ge& \frac{-B}{\tr_{\om} \om'} \sum_{i\le l} \left({1+ \varphi_{i \ol i}\over 1+ \varphi_{ l\ol l}}+ {1+ \varphi_{l \ol l}\over 1+ \varphi_{ i\ol i}}- 2\right) + \frac{\Delta f}{\tr_{\om} \om'}  \\
&\ge & \frac{\Delta f +Bn(n+1)}{\tr_{\om} \om'}- B \sum_{i} \frac{1}{1+\vp_{ i\b i}}\\
&\ge & \frac{\Delta f }{\tr_{\om} \om'}- B \, \tr_{\om'} \om\\
\end{eqnarray*}
which is exactly the content of the lemma.
\end{proof}

Now we may end the proof of Proposition \ref{prop}.

\begin{proof}[Proof of Proposition \ref{prop}]
We use lemma \ref{klem}, which gives:
\begin{equation}
\label{eq:lem2}
\Delta' \log \tr_{\om}\om' \ge \, \frac{\Delta (\psi_1-\psi_2)}{\tr_{\om}\om' }-C \,  \tr_{\om'}\om 
\end{equation}
if $\om'= \om + \iddb \vp$, and $\Delta'$ is the laplacian associated to $\om'$.
Let us observe the both following inequalities: $\om' \le \tr_{\om}\om'  \, \om$ and $C\om + \iddb \psi_2 \le \tr_{\om'}(C\om + \iddb \psi_2) \,  \om'$. They lead to 
\[n \le \tr_{\om}\om'  \,  \tr_{\om'}\om  \]
and
\[\frac{n+\Delta \psi_2}{\tr_{\om} \om'} \le C \tr_{\om'} (\om) + \Delta'  \psi_2\]
Therefore, \eqref{eq:lem2} gives
\[\Delta' (\log \tr_{\om}\om' +\psi_2) \ge \, -C_1 \,  \tr_{\om'}\om \]
with $C_1= (2+\frac 1 n )C$.\\
As $\om' = \om + \iddb \vp$, we have $n= \tr_{\om'} (\om) + \Delta' \vp$, so that
\[\Delta' (\log \tr_{\om}\om' +\psi_2-(C_1+1)\vp) \ge   \tr_{\om'}\om -C_2\]
with $C_2=n(C_1+1)$.\\
Now, at the point $p\in X$ where $\log \tr_{\om}\om' +\psi_2-(C_1+1)\vp$ attains its maximum, we have
\[ (\tr_{\om'} \om)(p) \le C_2\]
Moreover, from the inequality
\[ \tr_{\om} \om' \le e^{\psi_1-\psi_2} (\tr_{\om'} \om)^{n-1}\]
we deduce
\begin{eqnarray*}
\log(\tr_{\om} \om') &=& (\log \tr_{\om}\om' +\psi_2-(C_1+1)\vp)+(C_1+1)\vp - \psi_2\\
& \le  & (\log \tr_{\om}\om'+\psi_2-(C_1+1)\vp)(p)+(C_1+1)\vp - \psi_2\\
&\le & \psi_1(p)+(n-1)\log C_2 -(C_1+1)\vp(p) +(C_1+1)\vp - \psi_2 \\
& \le & C_3
\end{eqnarray*}
where $C_3=2C(C_1+2)+(n-1)\log C_2$. As $\om = e^{\psi_2-\psi_1} \om'$, one has 
\begin{eqnarray*}
\tr_{\om'} \om &\le& e^{\psi_2-\psi_1} (\tr_{\om} \om' )^{n-1}\\
&\le & e^{2C+(n-1)C_3}
\end{eqnarray*}
and therefore we may find a constant $A>0$ depending only on $n$ and $C$ such that
\[A^{-1} \om \le \om + \iddb \vp  \le A \,\om\]
which concludes the proof.
\end{proof}

\medskip
During the following paragraphs, we will construct the sequence of metrics
$(\ome)_{\varepsilon> 0}$ satisfying equation \eqref{eq:ma} with $\lambda\in \{0, 1\}$ and we will verify that it has the properties required in Proposition \ref{prop} (applied to $\psi_1=F_\varepsilon+\lambda (\vp_\varepsilon+ \psi_\varepsilon)$ and $\psi_2=0$); this will conclude the proof of the Main Theorem, by passing to the limit when $\ep \to 0$. More precisely, we need to show the following uniform estimates:
\begin{enumerate}
\item [$(i)$] $\sup_X  |\vpe| \le C$,
\item [$(ii)$] $\ome\ge C^{-1}\omega, \quad \sup_X|\psi_\varepsilon|\le C, \quad \sup_X |F_\varepsilon|\le C$,
\item [$(iii)$] $\Delta_{\ome} F_\varepsilon \ge -C$,
\item [$(iv)$] $\Theta_{\ome}(T_X) \ge -C \,  \ome \otimes \mathrm{Id}_{T_X}$ .
\end{enumerate}
Indeed, if these estimates are established, then the only missing point in order to apply Proposition \ref{prop} is the lower bound for $\Delta_{\ome}(\vpe+ \psi_{\ep})$. But as $\ome+ dd^c \vpe$ is a Kähler form and $\ome=\om + \iddb \psi_{\ep}$ satisfies $\ome \ge C^{-1} \om$, then 
$\Delta_{\ome}(\vpe+ \psi_{\ep}) \ge -n(1+C^{-1})$, as desired.\\

Finally, the case where $\lambda = -1$ will be studied in section \ref{logfano}, where Proposition \ref{prop} will be used in its general form, in contrary to the previous cases $\lambda\in \{0,1\}$ where we took $\psi_2=0$. 
\section{The sequence of metrics}
\label{sec2}
We consider the function $\chi$ such that for any $t\ge 0$, 
\begin{equation}
\label{eq:chi}
\chi(\varepsilon^2+ t)= {1\over \tau}\int_0^t{(\varepsilon^2+ r)^{\tau}- \varepsilon^{2\tau}\over r}dr
\end{equation}
and let $s$ be a holomorphic section of some hermitian line bundle $(L, h)$; 
we have
\[\dbar \big( \chi(\varepsilon^2+ |s|^2)\big)= 
\chi^\prime(\varepsilon^2+ |s|^2)\langle s, D^\prime s\rangle .\]
We apply the $\partial$-operator to the preceding inequality and we obtain
\[\ddbar \big( \chi(\varepsilon^2+ |s|^2)\big) = 
\chi^{\prime\prime}(\varepsilon^2+ |s|^2)|s|^2\langle D^\prime s, D^\prime s\rangle+ \chi^\prime(\varepsilon^2+ |s|^2)\langle D^\prime s, D^\prime s\rangle
- \chi^\prime(\varepsilon^2+ |s|^2)|s|^2\Theta \]
where $\Theta$ is the curvature form of $(L,h)$. By formula \eqref{eq:chi}, we obtain
$$\tau\chi^\prime(\varepsilon^2+ t)= {(\varepsilon^2+ t)^{\tau}- \varepsilon^{2\tau}\over t}$$
as well as
$$\tau \chi^{\prime\prime}(\varepsilon^2+ t)= 
{\tau(\varepsilon^2+ t)^{\tau-1}\over t}-
{(\varepsilon^2+ t)^{\tau}- \varepsilon^{2\tau}\over t^2}.$$
Hence we get
\begin{equation}
\label{metric}
\ddbar \big( \chi(\varepsilon^2+ |s|^2)\big)= (\varepsilon^2+ |s|^2)^{\tau-1}\langle D^\prime s, D^\prime s\rangle- {1\over \tau}
(\varepsilon^2+ |s|^2)^{\tau}- \varepsilon^{2\tau})\Theta.\\
\end{equation}

We indicate next the construction of a sequence of smooth K\"ahler 
metrics $\displaystyle(\ome)_{\varepsilon\ge 0}$, which 
is a regularization of the cone metric (cf. e.g. \cite{Clodo}).\\

For each $j\in J$, let $\chi_j= \chi_{j, \varepsilon}: [\varepsilon^2, \infty[\to \bR$ be the function defined as follows:  
\begin{equation}
\label{3}
\chi_j(\varepsilon^2+ t)= {1\over \tau_j}\int_0^t{(\varepsilon^2+ r)^{\tau_j}- \varepsilon^{2\tau_j}\over r}dr
\end{equation}
for any $t\ge 0$; we denote by $\varepsilon\ge 0$ a real number. There exists a constant $C> 0$ independent of $j, \varepsilon$ such that
\begin{equation}
\label{4}
0\le \chi_j(t) \le C
\end{equation}
provided that $t$ belongs to a bounded interval. Also, for each $\varepsilon> 0$
the function defined in \eqref{3} is smooth.\\

Let $(L_j, h_j)_{j\in J}$ be a set of hermitian line bundles, such that for each 
$j$ there exists a holomorphic section $s_j\in H^0(X, L_j)$ whose zero set is 
precisely $Y_j$ i.e. 
$$Y_j= (s_j= 0).$$ For each $j$ we denote by $h_j$ a smooth hermitian metric on the bundle $L_j$. 
The choice of the function $\chi_j$ as above is motivated by the following equality, derived from \eqref{metric}:
\begin{equation}
\label{5}
\sqrt{-1}\ddbar \chi_j\big(\varepsilon^2+ |s_j|^2\big)=
\sqrt{-1}{\langle D^\prime s_j, D^\prime s_j\rangle\over 
(\varepsilon^2+ |s_j|^2)^{1- \tau_j}}
- \frac{1}{\tau_j} \big((\varepsilon^2+ |s_j|^2)^{\tau_j}-\varepsilon^{2\tau_j}\big)\Theta_{h_j}(L_j)
\end{equation}
where we denote by $|s_j|$ the norm of the section $s_j$ measured with respect to the metric $h_j$, by $D^\prime$ the $(1, 0)$ part of the Chern connection associated to $(L_j, h_j)$, and by $\Theta_{h_j}(L_j)$ the curvature form of $(L_j, h_j)$.\\

Let $\omega$ be any K\"ahler metric on $X$; we consider the $(1,1)$-form 
\begin{equation}
\label{6}
\omega_{ \varepsilon}:= \omega+ {1\over N}\sum_{j\in J} dd^c \chi_j\big(\varepsilon^2+ |s_j|^2\big)
\end{equation}
on $X$, so that we have 
$$\displaystyle \psi_\varepsilon:= {1\over N}\sum_{j\in J}\chi_j\big(\varepsilon^2+ |s_j|^2\big).$$ 
It is a K\"ahler metric for each $\varepsilon> 0$ 
provided that $"N"$ above is positive enough constant,  
given equality \eqref{5}. If $\varepsilon= 0$, then we will denote the metric above by $\omega_o$. Setting $\Theta_j:=\Theta_{h_j}(L_j)$, we obtain

\begin{equation}
\label{met}
\ome= \omega+ \frac 1 N\sum_j\Big(
{\sqrt{-1}\langle D^\prime s_j, D^\prime s_j\rangle\over (\varepsilon^2+ |s_j|^2)^{1-\tau_j}}- {1\over \tau_j}
(\varepsilon^2+ |s_j|^2)^{\tau_j}- \varepsilon^{2\tau})\Theta_{j}\Big).
\end{equation}
and we remark that for each $\ep> 0$, the metric $\ome$ is smooth.
For $\ep= 0$, the metric $\omega_o$ has precisely the cone singularities 
induced by the pair $(X, D)$. In conclusion, the family 
$\displaystyle (\ome)_{\varepsilon> 0}$ can be seen as regularization of the 
metric $\omega_o$.
\section{Bounding the curvature from below}
\label{sec:curv}
To understand better the curvature estimate we have to prove, we consider a small open set $\Omega$
centered at $p\in X$, and $p\in \cap (s_j=0)$ for $j= 1,\ldots, d$ 
(these sets are assumed to have strictly normal intersections). We define the $(z)$-coordinates by the local expression of $s_j$,
completed in an arbitrary manner. Then we notice that we have
\[\omega_\varepsilon|_\Omega\ge C\sum_i\sqrt{-1}{dz^i \wedge dz^{\bar i}\over (\varepsilon^2+ |z^i|^2)^{1-\tau_i}}\]
hence for any vector 
\[v= \sum {v^i}{\partial\over \partial z^i}\]
such that $\displaystyle |v|_{\ome}= 1$ we have the estimate
\begin{equation}
\label{in7}
\big |{v^i}\big |^2
\le C(\varepsilon^2+ |z^i|^2)^{1-\tau_i}
\end{equation}
where the constant $C$ above is independent of $\varepsilon$.\\

Next, we perform the computation of the components of the curvature tensor;
the point is that only the ``diagonal" components 
$\displaystyle R_{\varepsilon p\ol p p\ol p}(z)$ blow up
as $\varepsilon \ll 1$, and they do it in the
right direction, so that \eqref{in2} will be satisfed. In order to reduce the
complexity of the evaluation of $R_{\varepsilon p\ol q r \ol s}(z)$, we will use
the coordinate systems which are constructed in the following paragraph.

\subsection{An appropriate coordinate system}
\noindent We have the following elementary result.

\medskip 

\begin{lemm}
\phantomsection
\label{lem:coor}
Let $(L_1, h_1),\ldots, (L_d, h_d)$ be a set of hermitian
line bundles, and for each index $j= 1,\ldots,d$, let $s_j$ be a section of $L_j$; we assume that the hypersurfaces 
$$Y_j:= (s_j= 0)$$
are smooth, and that they have strictly normal intersections. 
Let $\displaystyle p_0\in \cap Y_j$; then there exist a constant $C>0 $ and an open set $\Omega\subset X$ centered at $p_0$, such that for any point $p\in \Omega$ there exists a coordinate system $z= (z^1,\ldots, z^n)$ and a trivialization $\theta_j$ for $L_j$ such that:

\begin{enumerate}
\item [$(i)$] For $j= 1,\ldots, d$, we have $Y_j\cap \Omega= (z^j= 0)$;

\item [$(ii)$] With respect to the trivialization $\theta_j$, the metric
$h_j$ has the weight $\varphi_j$, such that 
\begin{equation}
\label{eq:coor}
\varphi_j(p)= 0, \quad d\varphi_j(p)= 0, \quad 
\Big|{\partial^{|\alpha|+|\beta|}\varphi_j\over \partial z^{\alpha}\partial z^{\bar \beta}}(p)\Big|\le C_{\alpha, \beta}
\end{equation}
for some constant $C_{\alpha, \beta}$ depending only on the multi indexes $\alpha, \beta$.
\end{enumerate}
\end{lemm}

\begin{proof}
It is completely standard, but we will nevertheless 
provide a detailed argument, for the sake of completeness. \\
Let $U\subset X$ be a coordinate open set, such that $p_0\in U$. We consider
the trivialization
$$\rho_j:L_j|_U\to U\times \bC$$
and let $w^1,\ldots, w^d$ be the expression of $s_1,\ldots, s_d$ 
with respect to the trivializations above. By the "simple normal crossing" hypothesis,
we can complete the $w's$ to a coordinate system at $p_0$. For each $j$, let
$\mu_j$
be the weight of the smooth metric $h_j$ when restricted to $U$. \\

If $p= p_0$, then we consider the Taylor expansion of $\mu_j$ at $p_0$, as follows
$$\mu_j(w)= \mu_j(0)+ 2\, \mathrm{Re}\Big(\sum_{i= 1}^nw^ia_{ij}\Big)+ {\rm higher\  order\ terms}$$
and we define $\theta_j$ by composing $\mu_j$ with the map
$$(w, v)\to (w, v e^{-l_{j, 0}(w)}),$$
where $l_{j, 0}(w):= {1\over 2}\mu_j(0)+ \sum_{i= 1}^nw^ia_{ij}.$
With respect to the new trivialization $\theta_j$, the section $s_j$ 
corresponds to $\displaystyle w^je^{-l_{j, 0}(w)}$, and the weight $\varphi_j$
of $h_j$ vanish at $p_0$, together with its differential.\\
We define $\displaystyle z^j:= w^je^{-l_{j, 0}(w)}$ for $j= 1,\ldots, d$ and $z^j:= w^j$
for $j\ge d+1$; then there exists $\Omega\subset U$ such that $z= (z^j)$
is a coordinate system on $\Omega$, and thus the lemma is established
for $p= p_0$.\\

To prove the result in full generality, we consider the Taylor expansion 
of $\mu_j$ at the point $p$
$$\mu_j(w)= \mu_j(p)+ 2\, \mathrm{Re}\Big(\sum_{i= 1}^n(w_i- p_i)a_{j}^i\Big)+ {\rm higher\  order\ terms}$$
and then we define $\theta_j$ by composing $\mu_j$ with the map
$$(w, v)\to (w, v e^{-l_{j}(w)}),$$
where $l_j(w):= {1\over 2}\mu_j(p)+ \sum_{i= 1}^n(w_i- p_i)a_{j}^i.$\\

The important observation is that $\Vert l_j- l_{j, 0}\Vert\ll 1$
provided that $\displaystyle \Vert p-p_0\Vert\ll 1$ is small enough, so that 
for a well-chosen open set $\Omega\subset U$ containing $p_0$, 
the functions $\displaystyle z_j:= w_je^{-l_j(w)}$ for $j= 1,\ldots, d$ and $z_j:= w_j$
for $j\ge d+1$ will still form a coordinate system, for any $p\in \Omega$.
The lemma is therefore completely proved, since the constant $C_{\alpha, \beta}$ in \eqref{eq:coor}
can be taken to be an upper bound for the coefficients of all the derivatives of $\mu_j$ of order $|\alpha|+|\beta|$, with respect to the $(z)$--coordinates
(see the definition of $z$ as a function of $w$ above).

Of course that the coordinates $z= (z^j)$ we are constructing here
are strongly dependent on the point $p$; nevertheless, what will be important in the following computations is the uniformity of the constant $"C"$.
\end{proof}

\subsection{Asymptotic of the metric}

For further use, we develop here the term 
$\langle D^\prime s, D^\prime s\rangle$; the local expression below is 
written with respect to the coordinates constructed in lemma \ref{lem:coor}.
If $s= s_j$ is a section of $L_j$, then we have
\[\langle D^\prime s, D^\prime s\rangle= e^{-\varphi_j}\Big(dz^j+ z^j\frac{\partial \vp_j}{\partial z^p}dz^p\Big)\wedge \Big(dz^{\b j}+ z^{\b j}\frac{\partial \vp_j}{\partial z^{\b p}}dz^{\b p}\Big).\]

In particular, we have the following explicit expression of the coefficients $(g_{p\b q })$ of the
metric $\ome$ constructed in section \ref{sec2}:
\begin{eqnarray}
g_{p\b q }&= &\, \omega_{p\b q}+ e^{-\varphi_q}{\delta_{pq}+ z^q\alpha_{pq}
\over (\ep^2+ |z^q|^2e^{-\varphi_q})^{1- \tau_q}}+ 
e^{-\varphi_p}{z^{\b p}\overline {\alpha_{qp}}
\over (\ep^2+ |z^p|^2e^{-\varphi_p})^{1- \tau_p}}+ \label{eq:metr} \\
&+ & \sum_k{|z^{k}|^2 {\beta_{kpq}}
\over (\ep^2+ |z^k|^2e^{-\varphi_k})^{1- \tau_k}}+
((\ep^2+ |z^k|^2e^{-\varphi_k})^{\tau_k}-\ep^{2\tau_k})\frac{\partial^2 \vp_k}{\partial z^p \partial \bar z^q} \nonumber
\end{eqnarray}
 
\noindent The expressions of $\alpha, \beta$ above are functions of the partial derivatives of $\varphi$; in particular, $\alpha$ is vanishing at the given point $p$
at order at least 1, and the order of vanishing of $\beta$ at $p$ 
is at least 2 (this helps a lot in the computations to follow...).\\

\noindent Before we begin to estimate the curvature of $\omega_{\ep}$, let us give an estimate for the coefficients of $(g^{p\b q})$: 

\begin{lemm}
\phantomsection
\label{lem:inv}
In our setting, and for $\ep^2+|z|^2$ sufficiently small, we have at the previously chosen point $p$:
\begin{enumerate}
\item[$(i)$] For every $k \in \{1, \ldots, d\}$, \[g^{k\b k}(p)=(\ep^2+|z^k|^2)^{1-\tau_k}(1+\O((\ep^2+|z^k|^2)^{1-\tau_k}));\]
\item[$(ii)$] For every $k,l \in \{1, \ldots, d\}$ such that $p\neq q$, \[g^{k \b l}(p)=\O((\ep^2+|z^k|^2)^{1-\tau_k}(\ep^2+|z^l|^2)^{1-\tau_l}),\]
\end{enumerate}
the $\O$ being with respect to $\ep^2+|z|^2$ going to zero.
\end{lemm}

\begin{proof}
Let us set $M(z)=(g_{k\b l})_{k\b l}$, and put it in the form:
\[M(z)=\left( 
\begin{array}{c|c} 
A & B \\ \hline 
B & C
\end{array}\right)\]
where $B$ and $C$ are bounded matrices (when $\ep^2+|z|^2$ goes to $0$), and 
\[A=\begin{pmatrix}
(\ep^2+|z^1|^2)^{\tau_1-1}+a_{11} & a_{12} & \cdots & a_{1d} \\
a_{21} & (\ep^2+|z^2|^2)^{\tau_2-1}+a_{22} & \cdots & a_{2d} \\
&&&\\
\vdots & & \ddots & \\
&&&\\
a_{d1} & \cdots & & (\ep^2+|z^d|^2)^{\tau_d-1}+a_{dd}
\end{pmatrix} \]
where all $a_{ij}$'s are bounded at $z=p$, \textit{uniformly} in $p$ and $\ep$. \\
If we denote by $M_{ij}$ the $(n-1, n-1)$ matrix obtained from $M$ by removing the $i$-th line and the $j$-th column, we can write, for $k\in \{1, \ldots , d\}$: 
\[\det M = (\ep^2+|z^k|^2)^{\tau_k-1} \det M_{kk}+\sum_{I\subset \{1, \ldots, \hat k, \ldots, d\}} A_I(z) \prod_{i\in I}(\ep^2+|z^i|^2)^{\tau_i-1}\]
where $A_I(z)$ is (uniformly) bounded at $z=p$.\\ 
But as $\ep^2+|z|^2$ goes to $0$, it is clear that \[\det M_{kk} \underset{\ep^2+|z|^2 \to 0}{\sim}  \det C \prod_{i\in \{1, \ldots, \hat k, \ldots, d\}} (\ep^2+|z^i|^2)^{\tau_i-1}\] and therefore: 
\[\fepp^{1-\dep} \frac{\det M}{\det M_{kk}} = 1+ \O((\ep^2+|z^k|^2)^{1-\tau_k}).\]
This shows $(i)$ by taking the inverse of the previous relation.\\
The second point is very similar, because 
\[\det M_{kl}  \underset{\ep^2+|z|^2 \to 0}{\sim} a_{kl}\,  \det C \prod_{i \in \{1,\ldots, d\} \setminus \{k,l\}}(\ep^2+|z^i|^2)^{\tau_i-1} \] and 
\[ \det M  \underset{\ep^2+|z|^2 \to 0}{\sim} \det C  \prod_{i \in \{1,\ldots, d\}}(\ep^2+|z^i|^2)^{\tau_i-1} .\] To be rigorous, we must note that if $a_{kl}=0$, the preceding relation might not hold, but we just have to replace the symbol $\sim$ by $\mathcal O$, which does not affect the final result.
\end{proof}

\subsection{Computation of the curvature: the non-diagonal case}
\label{ssec1}
In this subsection and the following one, we want to check that condition \eqref{in2} is fulfilled. To this end, 
we first show that there exists a constant $C>0$, independent of $\ep$, such that 

\[   (\ep^2+|z^p|^2)^{\frac{1-\tau_p}2}(\ep^2+|z^q|^2)^{\frac{1-\tau_q}2}(\ep^2+|z^r|^2)^{\frac{1-\tau_r}2}(\ep^2+|z^s|^2)^{\frac{1-\tau_s}2}
\big|R_{\varepsilon p \b q r \b s}(z)\big| \le C \\\\ \]
for any indexes $(p, q, r, s)$ not all equal.

\subsubsection{Estimating $\frac{\partial^2 g_{p\b q}}{\partial z^r \partial \bar z^s}$}

We use here the equality (\ref{eq:metr}) in order to evaluate the expression
$\displaystyle \frac{\partial^2 g_{p\b q}}{\partial z^r \partial \bar z^s}$ at the point $p$. First we observe that we have 
\begin{eqnarray}
\frac{\partial g_{p\b q}}{\partial z^r}&= &\!\! -e^{-\varphi_q} {\delta_{pq}+ z^q\alpha_{pq}\over \epq^{1-\tau_q}}{\varphi_{q, r}}+ e^{-\varphi_q}{\alpha_{pq}\delta_{qr}+ z^q\alpha_{pq, r}\over \epq^{1-\tau_q}}- \nonumber \\
&- &\!\!(1- \tau_q)e^{-2\varphi_q} {(\delta_{pq}+ z^q\alpha_{pq})(\delta_{qr}z^{\b q}- |z^q|^2\varphi_{q, r})\over \epq^{2-\tau_q}}- e^{-\varphi_p}\frac{z^{\b p}\overline{\alpha_{qp}}} {\epp^{1-\tau_p}}   \varphi_{p, r}+ \label{first_der}\\
&+ & e^{-\varphi_p}{z^{\b p}\overline{\alpha_{qp, \b r}}\over \epp^{1-\tau_p}}- (1- \tau_p)e^{-2\varphi_p}\frac{z^{\b p}\overline{\alpha_{qp}}(\delta_{pr}z^{\b p}- |z|^2\varphi_{p, r})}{ \epp^{2-\tau_p}}+\nonumber \\
&+ & \sum_k{|z^{k}|^2 {\beta_{kpq, r}}\over (\ep^2+ |z^k|^2e^{-\varphi_k})^{1- \tau_k}}+ \sum_k\tau_ke^{-\varphi_k} {\delta_{kr}z^{\b k}- |z^k|^2\varphi_{k, r}\over (\ep^2+ |z^k|^2e^{-\varphi_k})^{1- \tau_k}}\theta_k+\nonumber \\
&+ & (\ep^2+ |z^k|^2e^{-\varphi_k})^{ \tau_k}\theta_{k, r}+ \frac{\partial \omega_{p\b q}}{\partial z^r};
\nonumber 
\end{eqnarray}
here we denote e.g. by $\theta_{i, k}$ the $z^k$-partial derivative of the function
$\theta_i$ (which is the $(p, \b q)$-component of the curvature form $\Theta_i$ of $(L_i, h_i)$ in the chosen coordinates). The previous relation holds modulo terms whose vanishing order at 
$p$ is at least 2. Next, we apply the operator
$\displaystyle {\partial\over \partial z^{\b s}}$ to the equality above, and we 
evaluate the result at $p$; we obtain the following result:
\begin{eqnarray}
\frac{\partial^2 g_{p\b q}}{\partial z^r\partial z^{\b s}}(p)&= & -\frac{\delta_{pq} \, \vp_{q, r\bar s}} {(\ep^2+|z^q|^2)^{1-\tau_q}} + {\alpha_{pq, \b s}\delta_{qr}+ z^q\alpha_{pq, r \b s}\over \fepq^{1-\tau_q}}-(1-\tau_q) {|z^q|^2 \alpha_{pq, r}\delta_{qs}\over \fepq^{1-\tau_q}}- \nonumber  \\ 
&- &\!\!(1- \tau_q){|z^q|^2\alpha_{pq, \b s}\delta_{qr} +\delta_{qr}\delta_{qp}\delta_{qs}- |z^q|^2\varphi_{q, r\b s}\delta_{pq}\over \fepq^{2-\tau_q}}+  \label{second_der}  \\
&+& (1- \tau_q)(2- \tau_q){ |z^q|^2\delta_{qp}\delta_{qr}\delta_{qs}\over \fepq^{3-\tau_q}}+ {z^{\b p}\overline{\alpha_{qp, \b rs}}+ \delta_{ps}\overline{\alpha_{qp, \b r}}\over \fepp^{1-\tau_p}}- \nonumber\\
&- & (1- \tau_p){\delta_{pr}(z^{\b p})^2\overline{\alpha_{qp, s}}+ \delta_{ps}|z^{p}|^2\overline{\alpha_{qp, \b r}}\over \fepp^{2-\tau_p}}+\nonumber \\
&+ & \sum_k{|z^{k}|^2 {\beta_{kpq, r\b s}}\over (\ep^2+ |z^k|^2)^{1- \tau_k}}+ \sum_k\tau_k{\delta_{kr}\delta_{ks}- |z^k|^2\varphi_{k, r\b s}\over (\ep^2+ |z^k|^2)^{1- \tau_k}}\theta_k-\nonumber\\
&- & \sum_k\tau_k(1-\tau_k){|z^k|^2\delta_{ks}\delta_{kr}\over (\ep^2+ |z^k|^2)^{2- \tau_k}}\theta_k+ \sum_k\tau_k{\delta_{kr}z^{\b k}\over (\ep^2+ |z^k|^2)^{1- \tau_k}}\theta_{k, \b s}+\nonumber\\
&+& \sum_k\tau_k{\delta_{ks}z^k\over (\ep^2+ |z^k|^2)^{1- \tau_k}}\theta_{k, r}+ {\cO}(1).\nonumber\\
\nonumber
\end{eqnarray}

\medskip

\noindent When multiplied with the factor
\[  (\ep^2+|z^p|^2)^{\frac{1-\tau_p}2}(\ep^2+|z^q|^2)^{\frac{1-\tau_q}2}(\ep^2+|z^r|^2)^{\frac{1-\tau_r}2}(\ep^2+|z^s|^2)^{\frac{1-\tau_s}2} \]
the only unbounded term (as $\varepsilon\to 0$) in the expression \eqref{second_der}
is the following
\[-(1- \tau_q){\delta_{qr}\delta_{qp}\delta_{qs}\over \fepq^{2-\tau_q}}+ (1- \tau_q)(2- \tau_q){ |z^q|^2\delta_{qp}\delta_{qr}\delta_{qs}\over \fepq^{3-\tau_q}},\]
as one can see by a tedious but not difficult straightforward computation. The remaining diagonal case will be treated later.\\

\subsubsection{Estimating $g^{u \b t} \frac{\partial g_{p \b u}}{\partial z^r} \frac{\partial g_{t\b q}}{\partial \bar z^s}$} 

We have done half part of work by now. It remains to compute the second part of the curvature tensor given by:
\[g^{u\bar t}\frac{\partial g_{p \b u}}{\partial z^r}\frac{\partial g_{t\b q}}{\partial \bar z^s}\]
So we first compute the first derivatives of the metric at $p$:

\begin{eqnarray*}
\frac{\d g_{p\bar u}}{\d z^r}(p)&=& -(1-\tau_p) \frac{\delta_{pu}\delta_{ur}z^{\bar p}}{\fepp^{2-\tau_p}}+\frac{z^{\bar u} \alpha_{pu, r}}{\fepu^{1-\tau_u}}+\\
&+&\frac{z^{\bar p} \overline{\alpha_{up,\bar r}}}{\fepp^{1-\tau_p}}+\tau_r \frac{z^{\bar r}}{\fepr^{1-\tau_r}}\theta_r+\mathcal O(1)
\end{eqnarray*}
We want to estimate
\[\frac{\partial g_{p \b u}}{\partial z^r}(p)\frac{\partial g_{t\b q}}{\partial \bar z^s}(p)\] which is equal to:
\begin{eqnarray*}
&&\delta_{pu}\delta_{ur}\delta_{tq}\delta_{qs} (\tau_p-1)(\tau_s-1) \frac{\bar z^p z^q}{\fepp^{2-\tau_p}\fepq^{2-\tau_q}}+\\
&+&\delta_{pur} \frac{z^{\bar p} z^q}{\fepp^{2-\tau_p}\fepq^{1-\tau_q}}  \times \textrm{\small (bd)} +\delta_{pur} \frac{z^{\bar p} z^t}{\fepp^{2-\tau_p}\fept^{1-\tau_t}}  \times \textrm{\small (bd)} +\\
&+&\delta_{pur} \frac{ z^{\bar p}  z^s}{\fepp^{2-\tau_p}\feps^{1-\tau_s}}  \times \textrm{\small (bd)} +\delta_{tqs}\frac{z^{ q} z^{\bar u}}{\fepq^{2-\tau_q}\fepu^{1-\tau_u}}  \times \textrm{\small (bd)}+\\
&+&\delta_{tqs}\frac{z^{ q} z^{ \bar p}}{\fepq^{2-\tau_q}\fepu^{1-\tau_p}}  \times \textrm{\small (bd)}+\delta_{tqs}\frac{z^{ q} z^{\bar r}}{\fepq^{2-\tau_q}\fepr^{1-\tau_r}}  \times \textrm{\small (bd)}+\\
&+& \frac{ z^{\bar u}  z^q}{\fepu^{1-\tau_u}\feps^{1-\tau_q}}  \times \textrm{\small (bd)} + \frac{ z^{\bar u}  z^t}{\fepu^{1-\tau_u}\fept^{1-\tau_t}}  \times \textrm{\small (bd)} +\\
&+& \frac{ z^{\bar u}  z^s}{\fepu^{1-\tau_u}\feps^{1-\tau_s}}  \times \textrm{\small (bd)} +\frac{ z^{\bar p}  z^q}{\fepp^{1-\tau_p}\fepq^{1-\tau_q}}  \times \textrm{\small (bd)} +\\
&+& \frac{ z^{\bar p}  z^t}{\fepp^{1-\tau_p}\fept^{1-\tau_t}}  \times \textrm{\small (bd)} + \frac{ z^{\bar p}  z^s}{\fepp^{1-\tau_p}\feps^{1-\tau_s}}  \times \textrm{\small (bd)} +\\
&+& \frac{ z^{\bar r}  z^q}{\fepr^{1-\tau_r}\fepq^{1-\tau_q}}  \times \textrm{\small (bd)} + \frac{ z^{\bar r}  z^t}{\fepr^{1-\tau_r}\fept^{1-\tau_t}}  \times \textrm{\small (bd)} +\\
&+& \frac{ z^{\bar r}  z^s}{\fepr^{1-\tau_r}\feps^{1-\tau_s}}  \times \textrm{\small (bd)} +\frac{\partial g_{p \b u}}{\partial z^r}(p) \times \textrm{\small (bd)} + \frac{\partial g_{t\b q}}{\partial \bar z^s}(p) \times \textrm{\small (bd)}.\\
\end{eqnarray*}
where {\small{(bd)}} means “bounded terms".\\
Now we have to multiply with \[\fepp^{\frac{1-\tau_p}{2}} \fepq^{\frac{1-\tau_q}{2}} \fepr^{\frac{1-\tau_r}{2}} \feps^{\frac{1-\tau_s}{2}} g^{u \bar t }\] having in mind, thanks to lemma \ref{lem:inv} that \[g^{u \bar t }= \mathcal O(\fepu^{1-\tau_u}\fept^{1-\tau_t})\] whenever $u \neq t$, and \[g^{u \bar u }= \cO(\fepu^{1-\tau_u})).\]
Again, a straightforward computation (which will not be detailed here) 
shows then that each of those terms becomes bounded, except in the case $p=u=r=t=q=s$ with the term
\[(\tau_p-1)^2 \frac{ |z^p|^2}{\fepp^{2(2-\tau_p)}}\]
We may observe that we need the hypothesis $\tau_i \le 1/2$ for all $i$ in this computation.\\

This finishes to prove that the non-diagonal terms of the curvature of $\omega_{\ep}$ are uniformly bounded.

\subsection{Computation of the curvature: the diagonal case}
\label{ssec2}
By the considerations and computations of the previous section, we only have to show that the following expression is bounded from below when $\ep$ goes to zero:
\[\fepp^{2(1-\tau_p)}(1- \tau_p)\left(\frac{1}{\fepp^{2-\tau_p}}- \frac{(2- \tau_p)|z^p|^2}{ \fepp^{3-\tau_p}}+g^{p \bar p} \frac{(1-\tau_p)|z^p|^2}{\fepp^{2(2-\tau_p)}}\right)\]
whenever $p=q=r=s$. But this term equals:
\[(1-\tau_p)^2 \fepp^{-1-\dep}\left(\left[g^{p\b p}\fepp^{\dep-1}-1\right] |z^p|^2+\frac{\ep^2}{1-\dep} \right)\]
But by lemma \ref{lem:inv}, we know that $g^{p\b p}\fepp^{\dep-1}-1=\O(\fepp^{1-\dep})$, so that our expression becomes
\[(1-\tau_p)^2 \fepp^{-1-\dep}\frac{\ep^2}{1-\dep}+\O\left(\fepp^{1-2\dep}\right)\]
which is bounded \textit{from below} since $\dep \le 1/2$.\\

In conclusion, this shows that we have
$$|v^p|^2_{\ome}|w^p|^2_{\ome}R_{\varepsilon p \ol p p \ol p}\ge -C$$
for any vectors $v$ and $w$ of $\ome$-norm equal to 1, hence the relation
\eqref{in2} is established.

\subsection{\texorpdfstring{Estimating $\Delta_{\ome} \fe$}{Estimating the laplacian of Fe}}
\label{ssec3}
Let us fist fix some notations. $X$ is endowed with some Kähler form $\omega$, and for each $\ep>0$ small enough, we defined in \eqref{6} the Kähler form $\omega_{\ep}=\omega+dd^c \psi_{\ep}$ where $\psi_{\ep}= \frac 1 N \sum_j  \chi_{j,\ep}(\ep^2+|s_j|^2)$. 
As we saw in the introduction, we want to solve the following Monge-Ampère equation: 
\[(\ome+\iddb \vpe)^n = \frac{e^{f+\lambda (\vpe + \psi_{\ep})} }{\prod_{j=1}^d(\ep^2+|s_j|^2)^{1-\dej}} \om^n\]
for some function $f\in \mathscr C^{\infty}(X)$, and $\lambda \in \{0,1\}$.
Having in mind Proposition \ref{prop}, we may rewrite the last equation under the following form:
\[(\ome+\iddb \vpe)^n = e^{\fe+\lambda (\vpe + \psi_{\ep})} \ome^n\]
where \[\fe=f+\log \left(\frac 1 {\prod_{j=1}^d(\ep^2+|s_j|^2)^{1-\dej}} \frac{\om^n}{\ome^n} \right) \]
and, according to the inequality $(iv)$ occuring in the assumptions of Proposition \ref{prop}, we need to show that $\Delta_{\ome} \fe \ge -C$ for some constant $C>0$ independent of $\ep$. \\

First of all, we shall treat the term $\Delta_{\ome} f$. Indeed, from equality \eqref{met}, we easily deduce the existence of a constant $\gamma\in]0, 1[$ independent of $\ep$ such that 
\begin{equation}
\label{in:om}
\ome \ge \gamma \om
\end{equation}
Moreover, there exists $C>0$ such that $\iddb f \ge -C \om$ so that \[0 \le  \tr_{\ome}(\iddb f+C \om) \le \gamma^{-1}(Cn+ \tr_{\om}(\iddb f))\]
and thus
\[-C \gamma^{-1} \le \tr_{\ome}(\iddb f) \le  \gamma^{-1}(Cn+ \tr_{\om}(\iddb f)) \]
which shows that $\Delta_{\ome} f=\tr_{\ome}(\iddb f)$ is uniformly bounded.\\

Furthermore, it is rather easy to see that the term inside the logarithm, say $u_{\ep}$ may be written (in the previously chosen $(z)$-coordinates):
\[u_{\ep}=C(z)+ \sum_{I \subsetneq \{1, \ldots, n\} } A_I(z) \prod_I (\ep^2+|z^i|^2)^{1-\tau_i}\]
where $C(z)$ and $A_I(z)$ are sums of terms in the form \[B(z) \prod_{j\in J} \left[(\ep^2+|z_j|^2e^{-\vp_j})^{\tau_j}-\ep^{\dej}\right]\ \prod_{k\in K, K \cap I = \emptyset}\frac{z^k \alpha_k}{(\ep^2+|z^k|^2e^{-\vp_k})^{\lambda_k(1-\tau_k)}}\prod_{l\in L}\frac{|z^l|^2 \beta_l}{(\ep^2+|z^l|^2e^{-\vp_l})^{1-\tau_l}}\] where $B(z)$ is smooth independant of $\ep$, $\alpha_k$ is smooth and zero at $p$, $\beta_l$ is smooth and vanishes at order at least $2$ at $p$, and where $\lambda_k \in \{0, 1/2\}$. Indeed, the terms in the metric of the form  $\frac{z^k \alpha_k}{(\ep^2+|z^k|^2e^{-\vp_k})^{1-\tau_k}}$ may appear once or squared in the summands of the determinant $\ome^n$, but then we simplify them with the term $\prod_{i=1}^d (\ep^2+|z^i|^2)^{1-\tau_i}$ which is always possible, because the term $(\ep^2+|z^k|^2)^{-(1-\tau_k)}$ can't appear in those summands according to the alternating properties of the determinant. \\

So as to compute the laplacian of (the log of) this term, we do it pointwise, so that the first thing to do is using the normal coordinates $(w)$ in which $\ome$ is tangent at order $1$ to the flat metric. Indeed, in these coordinates, and at our point, we have for every $f>0$ and any real function $g,h$:
\[\Delta_{\ome} \log (f(w))=  \frac 1 f \Delta_{\ome}f(w) - \frac{1}{f^2} |\nabla f|_{\ome}^2\]
and
\[\Delta_{\ome} (gh) =  g \Delta_{\ome} h + h \Delta_{\ome} g + \la \nabla g, \nabla h\ra_{\ome} \]
where $\nabla f(w)= \sum \frac{\d f}{\d w^i} w^i $ is the usual (euclidian) gradient. \\\\

By lemma \ref{lem:inv}, we see that $u_{\ep}(p)$ admits a positive lower bound independent of $\varepsilon$ and in particular, as $u_{\ep}$ is clearly uniformly upper bounded, we obtain a constant $C>0$ independent of $\ep$ such that:
\begin{equation}
\label{uep}
|F_{\ep}| \le C
\end{equation}
The uniform minoration of $u_{\ep}(p)$ shows that it is enough to check that the functions 
appearing in the products have bounded $\ome$-laplacian and that their gradient's norm $|\nabla|^2_{\ome}$ is bounded too. Of course, we need the $(z)$-coordinates now, so we use the following rules:
\begin{eqnarray*}
\frac{\d^2 f}{\d w^i \d \bar w^i} &=& \sum_{k,l} \frac{\d z^k}{\d w^i} \overline{\frac{\d z^l}{\d w^i}} \frac{\d^2 f}{\d z^k \d \bar z^l}\\
\frac{\d f}{\d w^i}\overline{\frac{\d g}{\d w^i}} &=& \sum_{k,l} \frac{\d z^k}{\d w^i} \overline{\frac{\d z^l}{\d w^i}} \frac{\d f}{\d z^k}\overline{\frac{\d g}{\d z^l}}
\end{eqnarray*}
having still in mind that there exists some constant $C>0$ such that:
\[\left|\frac{\d z^k}{\d w^i} \right|^2\le C (\ep^2+|z^k|^2)^{1-\tau_k} .\]
Let us finally make the computations: 
\begin{eqnarray*}
\frac{\d^2 }{\d z^k \d \bar z^k}(\ep^2+|z^j|^2e^{-\vp_j})^{\alpha_j}(p)&=&\alpha_j \delta_{jk}\Big[(\ep^2+|z^j|^2)^{\alpha_j-1}+(\alpha_k-1) |z^j|^2 (\ep^j+|z^j|^2)^{\alpha_j-2}\Big] - \nonumber \\
&-& \alpha_j |z^j|^2 (\ep^2+|z^j|^2)^{\alpha_j-1} \vp_{j, k\bar k} \nonumber
\end{eqnarray*}
\[\left| \frac{\d }{\d z^k}(\ep^2+|z^j|^2e^{-\vp_j})^{\alpha_j}(p)\right|^2=\delta_{jk} \alpha_j^2 |z^j|^2 (\ep^2+|z^j|^2)^{2(\alpha_j-1)}\]
for $\alpha_k \in \{\tau_k, 1-\tau_k\}$.\\
The first term is a $\O\left((\ep^2+|z^k|^2)^{\alpha_k-1}\right)$ whereas the second is a $\O\left((\ep^2+|z^k|^2)^{2\alpha_k-1}\right)$. As $1-\tau_k+\alpha_k-1= \alpha_k-\tau_k \in \{0, 1-2\tau_k\}$ and $1-\tau_k+2\alpha_k-1=2\alpha_k - \tau_k \in \{\tau_k, 2-3 \tau_k \}$, and as $\tau_k \le 1/2$, our initial terms become bounded after being multiplied by $\left|\frac{\d z^k}{\d w^i} \right|^2$.\\ \\

Now we have to investigate the laplacians and gradient of $\frac{z^k \alpha_k}{(\ep^2+|z^k|^2e^{-\vp_k})^{\lambda_k(1-\tau_k)}}$. As $\lambda_k \in \{0, 1/2\}$ and $1-2\tau_k \ge 0$, $\frac{z^k}{(\ep^2+|z^k|^2)^{\lambda_k(1-\tau_k)}}$ is always bounded, so that the vanishing of $\alpha_k$ at $p$ guarantees that the gradient of our term is bounded. As for the second derivative, a computation leads to:
\begin{flushleft}
$\displaystyle \frac{\partial ^2}{\d z^r \d z^{\bar r}} \frac{z^k \alpha_k}{(\ep^2+|z^k|^2e^{-\vp_k})^{\lambda_k(1-\tau_k)}}(p) =$
\end{flushleft}
\[ \delta_{kr}\Big( \frac{\alpha_{k, \bar r}}{(\ep^2+|z^k|^2)^{\lambda_k(1-\tau_k)}}-\lambda_k(1-\tau_k)\frac{(z^k)^2 \alpha_{k,  r}+|z^k|^2 \alpha_{k, \bar r}}{(\ep^2+|z^k|^2)^{\lambda_k(1-\tau_k)+1}}\Big) + \frac{z^k \alpha_{k, r\bar r}}{(\ep^2+|z^k|^2)^{\lambda_k(1-\tau_k)}} 
\medskip \medskip \]

\noindent As we already observed, the last term on the right is bounded, and as for the other one, we may multiply it with $(\ep^2+|z^k|^2)^{(1-\tau_k)}$, so that they clearly become bounded, too.\\

To finish the computation, we have to deal with $\frac{|z^l|^2 \beta_l}{(\ep^2+|z^l|^2e^{-\vp_l})^{1-\tau_l}}$. As $\beta_l$ vanishes up to order $2$ at least, the gradient at $p$ is zero, and the second derivatives are equal to 
\[\frac{\partial ^2}{\d z^r \d z^{\bar r}} \frac{|z^l|^2 \beta_l}{(\ep^2+|z^l|^2e^{-\vp_l})^{1-\tau_l}}(p) =\frac{|z^l|^2 \beta_{l, r \bar r}}{(\ep^2+|z^l|^2)^{1-\tau_l}}\]
which are bounded.\\

This finishes the proof of the existence of a constant $C>0$ independent of $\ep$ such that $|\Delta_{\ome} \fe| \le C$. \\\\

We may now notice that inequalities \eqref{4}, \eqref{in:om} and \eqref{uep} show that the sought estimates $(ii)$ of Proposition \ref{prop} are satisfied. Therefore, according to this proposition, we are left to obtain $\cC^0$ estimates for $\vpe$ in order to get the Laplacian estimates.

\section{\texorpdfstring{From $\mathscr C^0$ to $\mathscr C^{2, \alpha}$ estimates: proof of the Main Theorem}{From C0 to C2+ estimates: proof of the Main Theorem}}
\label{sec:est}
\subsection{\texorpdfstring{$\mathscr C^0$ estimates}{C0 estimates}}
\label{subsec:est0}
We denote $\omega_{\ep}'=\omega_{\ep}+dd^c \vp_{\ep}$ the metric obtained by solving (in $\vp_{\ep}$) the Monge-Ampère equation
\[(\ome+dd^c \vp_{\ep})^n=e^{F_{\ep}+\lambda (\vpe + \psi_{\ep})} \ome^n \]
where $\lambda \in \{0, 1\}$. If $\lambda=0$, we choose $\vp_{\ep}$ to be the solution of the previous equation which satisfies
\[\int_X \vpe \,  dV_{\omega} =0\]\\
We are first wishing to get $\mathscr C^0$ estimates for $\vp_{\ep}$.\\

In the case where $\lambda=1$, we may follow e.g. \cite{Yau78}, \cite{Siu}; we will briefly recall next the argument. We choose a point $x_0$ where $\vpe$ atteins its maximum, so that $\det(\ome+ \iddb \vpe) \le \det(\ome)$ and thus $(F_{\ep}+\psi_{\ep})(x_0)+\vpe(x_0)\le 0$ and therefore $\sup_{X} \vpe \le \sup_X( |\fe|+|\psi_{\ep}| )\le M$, where $M>0$ is a constant which does not depend of $\ep$ as we already mentioned in \eqref{4} and \eqref{uep}. Moreover, we can apply the same argument with $\inf_X \vpe$, and we thus have our \textit{a priori} $\mathscr C^0$ estimates.\\

In the case where $\lambda=0$, we need a slightly more involved result coming from pluripotential theory. This is the following theorem 
basically due to S. Ko\l odziej
\cite{Kolo} (see also \cite{EGZ}, \cite{Zha}, \cite{SonTia}, \cite{Bl}, \cite{tos1}, \cite{tos2} for substantial improvements and interesting applications of this result):
\begin{theo}
Let $(X, \omega)$ be a compact K\"ahler manifold, and let 
$\displaystyle (F_\ep)_{\ep> 0}$ be a family of smooth functions, such that  
$\int_X e^{F_\ep} \om^n=\int_X \om^n$ and such that there exists 
two constants $\delta_0> 0$ and $C_0> 0$ for which we have
$$\int_Xe^{(1+ \delta_0)F_\ep} \om^n\le C_0,$$
for any $\varepsilon> 0$. 
Then any solution $\rho_\ep$ to the Monge-Amp\`ere equation
 \[(\om+\iddb \rho_\ep)^n = e^{F_\ep} \om^n\]
 satisfies 
 \[\mathrm{osc}\,  \rho_\ep:=\sup_X \rho_\ep - \inf_X \rho_\ep  \le C, \]
where $C>0$ depends only on $(X, \omega)$ and on $(\delta_0, C_0)$ above.
\end{theo}
We want to apply this result to the pair $(X, \omega)$; we remark that the equation 
$(\star_\varepsilon)$ can be written as
\[\om_{\rho_\varepsilon}^n=\frac{e^{f}}{\prod_{j=1}^d(\ep^2+|s_j|^2)^{1-\tau_j}}\om^n,\]
where $\displaystyle \rho_\varepsilon:= \psi_\ep+ \varphi_\ep$. By adjusting 
$\psi_\ep$ with a constant (whose variation with respect to $\varepsilon$ is bounded), we can assume that 
$$\displaystyle \int_X\rho_\varepsilon dV_\omega= 0.$$
By the previous theorem, we infer that
\[\sup_X |\rho_\varepsilon| \le \widetilde{C_0}\]
where $\widetilde{C_0}>0$ is a constant independent on $\ep$. Therefore we deduce our $\mathscr C^0$ estimate: 
\[\sup_X |\vpe| \le C_0.\] 
given the expression of $\psi_\varepsilon$ in \ref{sec2}.

\subsection{Laplacian estimates}

All we have to do is to combine the results of subsections \ref{ssec1}, \ref{ssec2}, \ref{ssec3} with Proposition \ref{prop}; actually, this proves the Main Theorem. For the sake of completeness, we shall explain in the next subsection why the solution $\omi$ to our equation $(\textrm{MA})$ (or equivalently $(\star_0)$) is smooth outside the support of $D$. As we already mentioned in the introduction, this result is already known.

\subsection{\texorpdfstring{$\mathscr C^{2, \alpha}$ estimates}{C2+ estimates}}

Remember that we wish to extract from the sequence of smooth metrics $\ome + \iddb \vpe$ some subsequence converging to a smooth metric on $\displaystyle X_0=X\setminus (\cup_{j\in J} Y_j)$.
In order to do this, we need to have \textit{a priori} $\mathscr C^k$ estimates for all $k$. The usual bootstrapping argument for the Monge-Ampère equation allows us to deduce those estimates from the $\mathscr C^{2,\alpha}$ ones for some $\alpha\in ]0,1[$. The crucial fact here is that we have at our disposal the following \textit{local} result, taken from \cite{Gilb} (see also \cite{Siu}, \cite[Theorem 5.1]{Bl}), which gives interior estimates. 
\begin{theo}
\phantomsection
\label{ek}

Let $u$ be a smooth psh function in an open set $\Omega \subset \CC^n$ such that $f:= \det(u_{i \bar j})>0$. Then for any $\Omega' \Subset \Omega$, there exists $\alpha \in]0, 1[$ depending only on $n$ and on upper bounds for $||u||_{\mathscr C^{0}(\Omega)}$, $\sup_{\Omega} \, \Delta \vp, ||f||_{\mathscr C^{0, 1}(\Omega)}, 1/\inf_{\Omega} \, f$, and $C>0$ depending in addition on a lower bound for $d(\Omega', \d \Omega)$ such that: 
\[||u||_{\mathscr C^{2, \alpha}(\Omega')}\le C.\]
\end{theo}

In our case, we choose some point $p$ outside the support of the divisor $D$, and consider two coordinate open sets $\Omega' \subset \Omega$ containing $p$, but not intersecting $\mathrm{Supp}(D)$. In that case, we may find a smooth Kähler metric $\omega_p$ on $\Omega$, and a constant $M>0$ such that for every $\ep>0$, we have $M^{-1} \om_p \le \omega_{\ep |\Omega} \le M \om_p$. Then one may take $u=\vpe$ in the previous theorem, and one can easily check that there are common upper bounds (i.e. independent of $\ep$) for all the quantities involved in the statement. This finishes to show the existence of uniform \textit{a priori} $\mathscr C^{2, \alpha}(\Omega')$ estimates for $\vpe$.\\

As we mentioned earlier, the ellipticity of the Monge-Ampère operator automatically gives us local \textit{a priori} $\mathscr C^k$ estimates for $\vpe$, which ends to provide a smooth function $\vp$ on $X_0$ (extracted from the sequence $(\vpe)_{\ep}$) such that $\omi=\om+\ddc \vp$ defines a smooth metric outside $\mathrm{Supp}(\D)$ satisfying 
\[ (\om+\ddc \vp)^{n} =\frac{e^{\lambda\vp+f}}{\prod_{j\in J}|s_j|^{2(1-\tau_j)}}\,\om^{n}. \]
Moreover, the strategy explained at the end of section \ref{sec:obs} and set up all along the section shows that this metric $\vp$ has cone singularities along $\D$, so this finishes the proof of the main theorem up to some detail. \\

Indeed, we have showed that on can extract a subsequence of $(\vpe)$ converging locally uniformly on $X\setminus \Supp(D)$ to a bounded function $\vp$ having cone singularities along $D$. We still need to make sure that $\vp$ is the solution of the Monge-Ampère equation (MA) we started from. Actually, this is a consequence of Ko\l odziej's stability theorem as the Monge-Ampère measures $\omega_{\vpe}^n$ converge in $L^p$ norm (for some $p>1$) to $e^{f+\lambda \vp}\mu_D$. We could also have used the following elementary argument: on $X\setminus \Supp(D)$, the local uniform convergence of $(\vpe)$ to $\vp$ ensures that $\om_{\vp}^n=e^{f+\lambda \vp} \mu_D$ on that set. Moreover, as $\vp$ is globally bounded on $X$, its Monge-Ampère does not charge any pluripolar set hence $\vp$ satisfies (MA) on the whole $X$, which concludes. 

\section{Kähler-Einstein metrics on log-Fano manifolds}
\label{logfano}

In this paragraph, we finish the program initiated in the previous sections showing that, roughly speaking, every Kähler-Einstein on a (log-smooth klt) pair $(X,\D)$ (in an appropriate meaning) has cone singularities along $\D$. We proved it when $K_X+\D$ ample or flat, and we will show it when $K_X+\D$ is anti-ample (that is $-(K_X+\D)$ ample).\\
So we have in mind the following equation: 
\[(\om+\ddc \vp)^n = e^{-\vp}\mu_{\D}\]
where \[\mu_{\D}= \frac{e^f}{\prod_{j\in J}|s_j|^{2(1-\tau_j)}} \om^n\] with the same notations as in the introduction. This volume form (or any equivalent volume form) will be called a cone volume form, $D$ being understood. \\

As is well-known, we cannot simply regularize this equation as in the previous sections and use \ref{prop} so as to obtain our solution, because there need not to be $\mathscr C^0$ estimates anymore. Even in the smooth case, the equation $(\om + dd^c \vp)^n = e^{f-\vp}\om^n$ may not have solutions, else there would always exist positively curved Kähler-Einstein metrics on Fano varieties, which is not the case. Therefore we will restrict ourselves to show that whenever a solution of the former Monge-Ampère equation exists, then it has cone singularities along the divisor.\\

Before we go any further, let us recall the definition of the class $\mathcal E(X,\om)$, which is composed of $\om$-psh functions $\vp$ such that their non-pluripolar Monge-Ampère $(\om + \iddb \vp)^n$ has full mass $\int_X \om^n$ (cf. \cite{GZ2}, \cite{BEGZ}). The main result of this section is the following:

\begin{theo}
\label{thm:logfano}
Let $X$ be a compact Kähler manifold and $\D=\sum (1-\tau_j) D_j$a divisor with simple normal crossings satisfying $0<\tau_j\le 1/2$, let $\mu_{\D}$ be a cone volume form, $\psi$ a bounded quasi-psh function, and $\omega$ a Kähler form on $X$. Then any solution $\vp \in \mathcal E(X,\om)$ of
\[(\om+ \ddc \vp)^n = e^{-\psi} \mu_{\D}\]
is Hölder-continuous and the metric $\om+ \ddc \vp$ has cone singularities along $\D$.
\end{theo}

\begin{proof}
We already know that $\vp$ is Hölder continuous: indeed as the rhs is in $L^p(dV_{\om})$ for some $p>1$, there exists a Hölder-continuous solution by Ko\l odziej's theorem (cf. \cite{Kolo,Kolo2}). But this equation admits a unique solution in $\mathcal E(X,\om)$ up to normalization, thanks to \cite[Theorem 3.3]{GZ2}, so that $\vp$ is necessarily Hölder continuous.

To control the singularities of the solution, one needs to write it as a limit of the regularized Monge-Ampère equations, which we will control thanks to the previous estimate \ref{prop}.
So we start by using Demailly's regularization theorem to approximate $\psi$ by a sequence $(\psi_{\ep})_{\ep}$ of bounded $2\omega$-psh functions, and we consider the following equation
\begin{equation}
\label{eqq2}
(\om+ \ddc \vp_{\ep})^n = e^{-\psi_{\ep}-G_{\ep}}\om^n 
\end{equation}
where $G_{\ep}= \sum_j a_j \log(|s_j|^2+\ep^2)$ (up to an additive normalizing constant guaranteeing that the previous equation has a solution; this constant is easily seen to be bounded when $\ep\to 0$), and $\vpe$ is the normalized solution given by Aubin-Yau's theorem. To study the asymptotic of $\om+\ddc \vpe$, one needs to compare this metric with the regularized cone metric $\ome= \om + \ddc \chi_{\ep}$ given in section \ref{sec2} (we changed the notations a bit). Therefore we rewrite equation \eqref{eqq2} in the form:
\begin{equation}
\label{eqq1}
(\om_{\ep}+ \ddc (\vp_{\ep}-\chi_{\ep}))^n = e^{-\psi_{\ep}+F_{\ep}}\ome^n 
\end{equation}
where $F_{\ep}=-\log\left( \frac{\prod_j (|s_j|^2+\ep^2)^{a_j}\ome^n}{\om^n}\right)$. Using the results of section \ref{ssec3}, we see that $\Fe$ is smooth and uniformly bounded so as its $\ome$-laplacian. Now we want to apply Proposition \ref{prop} with $\psi_1 =  F_{\ep}$ and $\psi_2= \psi_{\ep}$ and the Kähler form $\ome$, whose holomorphic bisectional curvature is known (thanks to section \ref{sec:curv}) to admit a uniform lower bound. As $\ome \ge \gamma^{-1} \om$ for some $\gamma >0$, $\psi_{\ep}$ is $2\gamma\ome$-psh and thus $\Delta_{\ome} \psi_{\ep} \ge -2n \gamma$. Therefore, it only remains to get $\mathscr C^0$ estimates on $\vpe$, for which are going to use Ko\l odziej's results. To this end, we use the form \eqref{eqq2} of our equation (else, the estimates we would obtain by Ko\l odziej's theorem could depend on $(X,\ome)$ and not just on a lower bound for the holomorphic bisectional curvature of $\ome$). Of course, equations \eqref{eqq1} and \eqref{eqq2} are \textit{the same equation}, but in \eqref{eqq2}, the reference metric $\om$ is fixed, and it is clear that $f_{\ep}:=e^{-\psi_{\ep}-G_{\ep}}$ satisfies $||f_{\ep}||_{L^p(dV_{\om})} \le M$ for some $p>1$ and $M>0$ not depending on $\ep$. We may now apply Ko\l odziej's result to obtain $C>0$ satisfying $\sup_X |\vpe| \le C$. As $\chi_{\ep}$ is uniformly bounded, then $\vpe-\chi_{\ep}$ is uniformly bounded, so that we have the desired $\mathscr C^0$ estimates for the solutions of \eqref{eqq1}.\\

Therefore, we may now apply Proposition \ref{prop}, and we get a constant $A>0$ independent of $\ep$ such that
\[A^{-1}  \ome \le \om + \ddc \vpe  \le A  \, \ome.\]
By Evans-Krylov theory (cf. theorem \ref{ek}), we then get interior (this means here local and outside $\Supp(\D)$) estimates at any order. Therefore we may extract from $\vpe$ a subsequence converging weakly to $\vp_{\infty}$ on $X$, such that $\vp_{\infty}$ is smooth outside $\D$, and satisfies
 \[(\om+ \ddc \vp_{\infty})^n = e^{-\psi}\mu_{\D} \]
 and also
 \[A^{-1}  \om_0 \le \om + \ddc \vp_{\infty}  \le A  \om_0\]
 where $\om_0$ is a metric with cone singularities.
As $\vp_{\infty}$ is bounded thanks to the $\mathscr C^0$ estimates, $\vp \in \mathcal E(X,\om)$ and 
 \[(\om+ \ddc \vp_{\infty})^n = (\om+ \ddc \vp)^n \]
so the uniqueness theorem of Guedj-Zeriahi \cite[Theorem 3.3]{GZ2} allows us to conclude that $\vp_{\infty}-\vp$ is constant. This finishes to prove that $\om + \ddc \vp$ has cone singularities along $\D$. \\
\end{proof}

Let $X$ be a smooth complex projective variety, $\D= \sum (1-\tau_j) \D_j$ an effective $\R$-divisor with simple normal crossing support and such that for all $j$, $\tau_j\in ]0,1[$. We choose sections $s_j$ of hermitian line bundles $(L_j,h_j)$ such that $\D_j=\mathrm {div} (s_j)$. If $-(K_X+\D)$ is ample, we call the pair $(X,\D)$ a \textit{log-Fano manifold}. \\
Under this context, we propose the following notion of Kähler-Einstein metric for $(X,\D)$:
\begin{defi}
\label{def:ke}
Let $(X,\D)$ be a log-Fano manifold. A Kähler-Einstein metric for $(X,\D)$ is defined to be a Kähler metric $\om$ on $X_0$ satisfying the following properties:
\begin{enumerate}
\item[$\bullet$] $\Ric \om =  \omega$ ;
\item[$\bullet$] There exists $C>0$, and $\mu_{\D}$ a cone volume form on $X_0$ such that:
\[ C^{-1} \mu_{\D} \le \om^n \le  C \mu_{\D}  \]
\end{enumerate}
\end{defi}

Rephrasing a result of Berman appearing in an earlier version of \cite{rber}, we have the following link between Kähler-Einstein metrics and Monge-Ampère equations, which shows that our definition of Kähler-Einstein metric for a log-Fano manifold coincides with the general definition for a pair as given by \cite{BBEGZ}:

\begin{prop}
\label{prop:ke2}
Let $(X,\D)$ be a log-Fano manifold. We choose a Kähler metric $\om_0\in -c_1(K_X+\D)$. Then any Kähler-Einstein metric $\om$ for $(X,D)$
extends to a Kähler current $\om=\om_0+\ddc \vp$ on $X$ where $\vp \in PSH(X,\om_0)\cap L^{\infty}(X)$ is a solution of 
\[(\om_0+\ddc \vp)^n = e^{-\vp-\vp_{\D}}\om_0^n\]
where $\vp_{\D}=\sum _{j \in J} (1-\tau_{j}) \log |s_{j}|^2+f$ for some $f\in \mathscr C^{\infty}(X)$. 
\end{prop}

\begin{proof} For the convenience of the reader, we give the proof of the proposition in terms of $\om$-psh functions.
Let us define a smooth function $\psi$ on $X_0$ by:
 \[\psi:=-\log \left( \frac{\prod_{j\in J} |s_j|^{2(1-\tau_j)} \om^n}{\om_0^n} \right)\]
Wet set $\Theta(D):=\sum a_i \Theta(D_i)$, where $\Theta(D_i)$ is the curvature of $(L_i, h_i)$. By assumption, $\psi$ is bounded on $X_0$, and we know that on this set, 
\[\ddc \psi = \om-\Ric \om_0^n  +\Theta(\D)\]
so that $\psi$ is $M\om_0$-psh for some $M>0$ big enough. As it is upper bounded, it extends to a (unique) $M\om_0$-psh function on the whole $X$, which we will also denote by $\psi$. 
Let now $f$ be a smooth potential on $X$ of $\Ric \om_0^n -\om_0-\Theta(\D)$. It is easily shown that $\vp:=\psi+f$ satisfies $\om_0+\ddc \vp = \om$ on $X_0$. Moreover $\vp \in PSH(X,\om_0)\cap L^{\infty}(X)$. Therefore its Monge-Ampère $(\om_0+\ddc \vp)^n$ does not charge any pluripolar set (and in particular it does not charge $\D$), and thus it satisfies the equation
\begin{eqnarray*}
(\om_0+\ddc \vp)^n&=&\frac{e^{-\vp+f}\om_0^n}{\prod_{j\in J} |s_{j}|^{2(1-\tau_j)}}\\
&=&e^{-\vp-\vp_{\Delta}}\om_0^n
\end{eqnarray*}
on the whole $X$, with the notations of the statement. 
\end{proof}

\begin{coro}
On a log-Fano manifold $(X,\D)$ such that the coefficients of $\D=\sum (1-\tau_j) D_j$ satisfy $0<\tau_j\le 1/2$, any Kähler-Einstein metric has cone singularities along $\D$. 
\end{coro}

\begin{proof}
Indeed, by the previous proposition, such a metric extends to a current $\om=\om_0+\ddc \vp$ on $X$ where $\vp \in PSH(X,\om_0)\cap L^{\infty}(X)$ is a solution of 
\[(\om_0+\ddc \vp)^n = e^{-\vp-\vp_{\D}}\om_0^n\]
Therefore, applying theorem \ref{thm:logfano} to $\psi=\vp$ and $\mu_{\D}=e^{-\vp_{\D}}\om_0^n$, we are done. 
\end{proof}

Let now $X$ be a Fano manifold, and $D\in |-K_X|$ a smooth divisor. By computing the $\alpha$-invariant of the pair $(X,(1-\tau)D)$, Berman has showed in \cite[Theorem 5.1]{rber} that the pair $(X,(1-\tau)D)$ admits a (weak) Kähler-Einstein cone metric $\om_{\tau}$ (along $D$) as soon as the angle $\tau$ is small enough. Therefore, applying the last corollary, we deduce that those metrics have actually cone singularities along $D$. This is the first step of the program initiated by Donaldson \cite{Don} to solve the so called Yau-Tian-Donaldson conjecture.

\section{Metrics with prescribed Ricci curvature on geometric orbifolds}

We consider a function $f\in \cC^\infty (X)$ and $\lambda\in \{0, 1\}$.
For each $\varepsilon> 0$, we denote by $\ome$ the metric constructed in section \ref{sec2}, and we consider the following Monge-Amp\`ere equation 
$$(\omega_{\varepsilon}+ dd^c \varphi_\varepsilon)^n=
{e^{f+ \lambda(\varphi_\varepsilon+\psi_{\ep})}\over \prod_{j\in J}(\varepsilon^2+ |s_j|^2)^{1- \tau_j}}\omega^n\leqno(\star_{\ep})$$
It has a unique solution $\varphi_\varepsilon\in \cC^\infty(X)$ 
if $\lambda= 1$, and it has a unique solution $\varphi_\varepsilon\in \cC^\infty(X)$ such that 
$$\int_X\varphi_\varepsilon \, dV_{\omega}=0$$
if $\lambda= 0$, respectively (cf. \cite{Aubin}, \cite{Yau78}). \\

We are going to apply Theorem \textbf{A}, by 
choosing the data in $(\star_{\ep})$ 
according to the hypothesis of Theorem \textbf{C}.
At first we assume that we have 
\begin{equation}
\label{7}
c_1(K_X+ D)= \lambda \{\omega\}
\end{equation}
for some K\"ahler metric $\omega$ on $X$; as above, $\lambda$ is equal to zero 
or 1. Then there exists a function $f\in C^\infty(X)$ such that 
$$\Theta_{\det (\omega)}(K_X)+ \sum_{j\in J}(1- \tau_j)\Theta_{h_j}(L_j)
+ dd^c f= \lambda\omega$$
as a consequence of the relation \eqref{7}.
Equality $(\star_0)$ shows that we have
\begin{equation}
\label{8}
-\Ric({\omega_{\infty}})= \lambda \, \omega_{\infty}
\end{equation}
on $X_0$, and by Theorem \textbf{C},  we also know 
the behaviour of $\omega_{\infty}$
near the support of the divisor $D$.\\

\noindent We assume next that there exists a closed, semi-positive definite form $\alpha$ on $X$ such that
\begin{equation}
\label{9}
c_1(K_X+ D)= \lambda^\prime \{\alpha\};
\end{equation}
where $\lambda^\prime\in \{-1, 1\}$.
Then we consider a function $f\in C^\infty(X)$ such that 
$$\Theta_{\det (\omega)}(K_X)+ \sum_{j\in J}(1- \tau_j)\Theta_{h_j}(L_j)
+dd^c f= \lambda^\prime\alpha$$
 which we plug in the family of equations $(\star_\varepsilon)$ together with
 $\lambda= 0$.   
The metric $\omega_{\infty}$ given by Theorem \textbf{A} satisfies 
\begin{equation}
\label{10}
-\Ric({\omega_{\infty}})= \lambda^\prime\alpha
\end{equation}
on $X_0$.

\noindent In conclusion, under the hypothesis of Theorem \textbf{C},  we can construct 
a K\"ahler metric $\omega_\infty$ on $X_0$, such that:

\begin{enumerate}
\item The Ricci curvature of $\omega_\infty$ satisfies equality \eqref{8} if $c_1(K_X+ D)$ is strictly positive or zero.

\item The Ricci curvature of $\omega_\infty$ satisfies equality \eqref{10} if the Chern class $c_1(K_X+ D)$ is semi-positive or semi-negative.

\item In both cases above, there exists a positive constant $C$ such that 
$$C^{-1}\omega_{o}\le \omega_{\infty}\le 
C\omega_{o}$$ 
at each point of $X_0$, and where $\omega_o$ is a metric with cone singularities along $D$.

\end{enumerate}

\section{Holomorphic tensors and the \# operator}
\label{sec3}
\noindent Let $M$ be a complex manifold; classically, one defines the 
tensor fields which are contravariant of degree $r$ and covariant of degree $s$
as follows
\begin{equation}
\label{11}
T^r_sM:= \big(\otimes ^rT_M\big)\otimes \big(\otimes ^sT_M^\star\big).
\end{equation}
In our present context, we consider $M:= X_0$, that is to say the
Zariski open set $\displaystyle X\setminus \cup_jY_j$, and following \cite{Camp2}, we will be concerned with the tensor fields respecting the orbifold structure $(X, D)$. We first 
recall the definition of these objects as given in \cite{Camp2}, and then we will recast it in a more 
differential-geometric framework (which is better adapted for the computations in section \ref{sec4}). 

Let $x\in X$ be a point; since the hypersurfaces $(Y_j)$ have strictly normal intersections, there exist a small open set $\Omega\subset X$, together with a coordinate system $z= (z^1,\ldots, z^n)$ 
centered at $x$ such that $$Y_{j}\cap \Omega= (z^j= 0)$$ for $j= 1,\ldots, d$ and 
$$Y_j\cap \Omega= \emptyset$$ for the others indexes. Such a coordinate system will be called \emph{adapted} to the pair $(X, D)$ at $x$. 

We define the locally free sheaf $T^r_s(X|D)$ generated as an $\cO_X$-module by the tensors 
$$z^{\lceil (h_I- h_J)\cdot a\rceil}{\partial\over \partial z^{\otimes I}}\otimes dz^J$$
where the notations are as follows:

\begin{enumerate}

\item $I$ (resp. $J$) is a collection of positive integers in $\{1,\ldots, n\}$ of cardinal 
$r$ (resp. $s$) (we notice that we may have repetitions among the 
elements of $I$ and $J$, and we count each element according to its multiplicity).

\item For each $1\le i\le n$, we denote by 
$h_I(i)$ the multiplicity of $i$ as element of the collection $I$.

\item For each $i= 1,\ldots, d$ we have $a_i:= 1-\tau_i$, and
$a_i= 0$ for $i\ge d+ 1$.

\item We have 
$$z^{\lceil (h_I- h_J)\cdot a\rceil}:= \prod_i {(z^i)}^{\lceil (h_I(i)- h_J(i))\cdot 
a_i\rceil}$$

\item If $I= (i_1,\ldots, i_r)$, then we have
$${\partial\over \partial z^{\otimes I}}:= {\partial\over \partial z^{i_1}}\otimes \cdots
\otimes {\partial\over \partial z^{i_r}}$$
and we use similar notations for $dz^J$.

\end{enumerate}

\noindent Hence the holomorphic tensors we are considering here have
prescribed zeroes/poles near $X\setminus X_0$, according to the 
multiplicities of $D$. For reasons which will appear clearly in a moment, 
we have to consider a version of this definition, as follows.
\medskip

Let $g$ be any K\"ahler metric on $X_0$ with cone singularities along $D$:  
there exists a positive constant $C$ with the property that
\begin{equation}
\label{12}
C^{-1}\omega_o\le g\le C \, \omega_o.
\end{equation}
where we recall that $\omega_{o}$ is the metric with cone singularities along $D$ constructed in section \ref{sec2}.\\
The metric $g$ induces a metric on the vector bundle
$T^r_s(X_0)$ in a natural way; we introduce the following class of tensors on
$X_0$.

\begin{defi} Let $u$ be a smooth section of the bundle $T^r_s(X_0)$; we say that $u$ is bounded if there exists a constant $C> 0$ such that 
\begin{equation}
\label{13}
|u|_{g, x}^2\le C
\end{equation}
at each point $x$ of the Zariski open set $X_0$. We denote the space of bounded sections of $T^r_s(X_0)$ by 
$\cC^\infty_{\cB}\big(X_0, T^r_s(X_0)\big)$.

\end{defi}

\noindent We remark that the notion above is independent of the choice of the metric $g$, provided that the aforesaid metric satisfies the relation \eqref{12}.

\medskip

With respect to any coordinate system $z= (z^1,\ldots, z^n)$ we can write locally
$$g= {\sqrt {-1}\over 2}\sum_{\alpha, \beta}g_{\alpha \overline \beta }dz^\alpha\wedge 
d{z}^{\overline \beta}$$
and then $u$ can be expressed as follows
$$u= \sum_{I, J}u^I_J{\partial\over \partial z^{\otimes I}}\otimes dz^J; $$
the condition \eqref{13} becomes
\begin{equation}
\label{14}
\sum_{I, J, K, L}u^I_J\overline{u^K_L}g^{J\overline L}_{I\overline K} \, \le C.
\end{equation}
In inequality \eqref{14} above, we use the notations 
$$g^{J\overline L}_{I\overline K}:= g^{j_1\overline l_1}\ldots
g^{j_s\overline l_s}g_{i_1\overline k_1}\ldots g_{i_r\overline k_r}$$
and $(g^{p\overline q})$ is the inverse of the matrix 
$(g_{\alpha \overline \beta})$. 
\medskip

\noindent The following result explains the link between the global sections of 
the bundle $T^r_s(X|D)$ and the bounded sections of $T^r_s(X_0)$.

\begin{lemm}
\phantomsection
\label{lem32} 
Let $u$ be a smooth section of the bundle $T^r_s(X_0)$. Then 
$u$ corresponds to a holomorphic section of $T^r_s(X|D)$ if and only if 
$\dbar u= 0$ and $u\in \cC^\infty_{\cB}\big(X_0, T^r_s(X_0)\big)$. 
\end{lemm}

\begin{proof} The direct implication is immediate: 
let $z$ be a coordinate system adapted to $(X, D)$ at some point; we assume that its components are defined on some open set $\Omega\subset X$. 
We remark 
that by relation \eqref{12} the following inequalities
\begin{equation}
\label{15}
{C^{-1}\over |z^\alpha|^{2(1- \tau_\alpha)}}\le 
g_{\alpha \overline \alpha}(z)\le {C\over |z^\alpha|^{2(1- \tau_\alpha)}},
\end{equation}

\begin{equation}
\label{16}
|g_{\alpha \overline \beta}(z)|\le {C\over |z^\alpha|^{(1- \tau_\alpha)}
|z^\beta|^{(1- \tau_\beta)}},
\end{equation}

as well as 
\begin{equation}
\label{17}
{C^{-1}|z^\alpha|^{2(1- \tau_\alpha)}}\le 
g^{\alpha \overline \alpha}(z)\le {C |z^\alpha|^{2(1- \tau_\alpha)}}
\end{equation}
and

\begin{equation}
\label{18}
|g^{\alpha \overline \beta}(z)|\le {C |z^\alpha|^{(1- \tau_\alpha)}
|z^\beta|^{(1- \tau_\beta)}}
\end{equation}
hold pointwise on $\Omega$, where $C> 0$ is a constant.\\

Let $u$ be a holomorphic section of the bundle $T^r_s(X|D)$; then by definition,
its coefficients $u^I_J$ satisfy the inequality
$$|u^I_J|^2\le {C |z|^{2\lceil (h_I- h_J)\cdot a\rceil}}
$$
and by inequalities \eqref{15}-\eqref{18} we have 
\begin{equation}
\label{19}
\sum_{I, J, K, L}u^I_J\overline{u^K_L}g^{J\overline L}_{I\overline K}
\le C\sum_{I, J, K, L}|z|^{\lceil (h_I- h_J)\cdot a\rceil}|z|^{\lceil (h_K- h_L)\cdot a\rceil}|z|^{(h_J- h_I+ h_L- h_K)\cdot a}.
\end{equation}
The last expression is clearly bounded
(as $\lceil a\rceil \ge a$ for any real number $a$), and since $\dbar u= 0$ on $X_0$, our 
direct implication is established.\\ 

\noindent We consider next a bounded, holomorphic section $u$ of the bundle $T^r_s(X_0)$. The ``bounded" hypothesis combined with the above relations \eqref{15}-\eqref{18}
shows in particular that we have
$$|u^I_J|^2g^{J\overline J}_{I\overline I}\le C$$
hence we infer the inequality
$${|u^I_J|^2\over |z|^{2\lceil (h_I- h_J)\cdot a\rceil}}\le C
|z|^{2(h_I- h_J)\cdot a- 2\lceil (h_I- h_J)\cdot a\rceil}.$$
Then the holomorphic function 
$$\displaystyle{u^I_J\over z^{\lceil (h_I- h_J)\cdot a\rceil}}$$
is $L^2$-integrable, so it 
extends across the support of the divisor $D$, hence the lemma is proved.
We remark that this kind of arguments are very useful in the 
orbifold context, see \cite{Camp1} for other applications.

\end{proof}

\medskip

\noindent We recall next the definition of the $\#$-operator, and show that the induced map
$$\#: \cC^\infty_{\cB}\big(X_0, T^r_s(X_0)\big)\to \cC^\infty_{\cB}\big(X_0, T^s_r(X_0)\big)$$
is well defined.

\noindent To this end, if $v$ is a (local) section of the bundle $T^1_0(X_0)$, then we have
$$\# v:= \sqrt {-1} \overline v\!\mathrel{\lrcorner}\!g$$
Actually, this operator induces a map (still denoted by $\#$) between 
the tensor bundles 
$$\# : T^r_s(X_0)\to T^s_r(X_0)$$
which we describe next in local coordinates.

\noindent Let $z= (z^1,\ldots, z^n)$ be a system of local coordinates, such that 
$$g= {\sqrt {-1}}\sum_{\alpha, \beta}g_{\alpha \overline \beta }dz^\alpha\wedge 
d{z}^{\overline \beta};$$ 
the vector $v$ can be expressed as
$$v= \sum_\alpha v^\alpha{\partial\over\partial z^\alpha}$$
and then we have
\begin{equation}
\label{20}
\#v:= \sum_{\alpha, \beta}\overline {v^\beta}g_{\alpha \overline\beta}dz^\alpha.
\end{equation}
In a similar fashion one can see that the $\#$-operator (or rather its inverse)
acts on $T^0_1(X_0)$ as follows: 
if $\rho= \sum_i \rho_idz^{i}$ is a local section of the bundle 
above, then we have
\begin{equation}
\label{21}
\# \rho= \sum_{p, i}\overline{\rho_p}g^{i\overline p}{\partial\over \partial z^i}.
\end{equation}

\noindent We extend the definition of $\#$ to any element of $T^r_s(X_0)$
as indicated here
$$\#\Big({\partial\over \partial z^{\otimes I}}\otimes dz^{\otimes J}\Big):=
\#(dz^{j_1})\otimes \cdots \otimes \#(dz^{j_s})\otimes \#\Big({\partial\over \partial z^{i_1}}\Big)
\otimes \cdots \otimes \#\Big({\partial\over \partial z^{i_r}}\Big)$$
and we recall next the following basic fact (which will nevertheless play an important role in what follows).

\begin{lemm}
\phantomsection
\label{lem33}
The operator $\#$ induces a well-defined map
$$\#: \cC^\infty_{\cB}\big(X_0, T^r_s(X_0)\big)\to \cC^\infty_{\cB}\big(X_0, T^s_r(X_0)\big)$$
between the spaces of bounded tensors.
\end{lemm}
\begin{proof} Let $u$ be a bounded tensor; by definition, there exists a constant $C> 0$ such that 
$$\sup_{x\in X_0}|u|_{g, x}\le C< \infty$$
and it is well-known that the $\#$-operator is an isometry
(this can be checked directly, given the local expressions \eqref{20} and \eqref{21} above), hence
$$\sup_{x\in X_0}|\#(u)|_{g, x}= \sup_{x\in X_0}|u|_{g, x}\le C< + \infty$$
and the lemma is proved. 
\end{proof}

\begin{rema}{\rm The vanishing conditions we impose to the elements of
$T^r_s(X|D)$ are not preserved by the operator $\#$ in general; this is the main reason why we have to consider the class of bounded tensors. }
\end{rema}

\section{Bochner formula and cut-off procedure}
\label{sec4}
We consider here a K\"ahler metric $g$ on $X_0$, satisfying condition \eqref{12}; in other words, $g$ has the same order zero asymptotic
as the cone metric $\omega_o$
near the divisor $D$.  

Let $x\in X_0$ be an arbitrary point. Since $g$ is a K\"ahler metric, there exists a 
geodesic coordinate system $z= (z^1,\ldots , z^n)$ centered at $x$, such that if we write
\begin{equation}
\label{22}
g= \sqrt{-1}\sum_{i, j}g_{j \overline i}dz^j\wedge dz^{\overline i}
\end{equation}
then the coefficients in the expression \eqref{22} above have the following expansion up
to order 3:
$$g_{j \overline i}(z)= \delta_{ij}- \sum_{k, l}R_{j\overline i k\overline l}z^kz^{\overline l}+ \cO(|z|^3).$$

\noindent We recall here the following Bochner type formula; see
\cite{BY}, \cite[Lemma 14.2]{Dem95}.

\begin{lemm}
\phantomsection
\label{lem41}
Let $u$ be a tensor of $(r, s)$-type on $X_0$, with compact support in $X_0$. Then we have
$$\int_{X_0}|\overline\partial(\# u)|^2dV_g= \int_{X_0}|\overline\partial u|^2dV_g
+ \int_{X_0}\langle \cR(u), u\rangle dV_g,$$
where $\cR$ is an order zero operator, defined as follows. We write 
\begin{equation}
\label{23}
u= \sum_{I, J}u^I_J{\partial\over \partial z^{\otimes I}}\otimes dz^{\otimes J}
\end{equation}
and then we have

\begin{align*}
\cR(u):= & \sum_{I, J, p, l}u^I_JR_{i_p\overline l}
{\partial\over \partial z^{i_1}}\otimes \cdots \otimes
{\partial\over \partial z^{i_{p-1}}}\otimes
{\partial\over \partial z^{l}}\otimes
{\partial\over \partial z^{i_{p+1}}}\otimes \cdots \otimes
{\partial\over \partial z^{i_r}}\otimes dz^{\otimes J}-\\
- & \sum_{I, J, q, k}u^I_JR_{j_q\overline k}
{\partial\over \partial z^{\otimes I}}\otimes
dz^{j_1}\otimes \cdots \otimes dz^{j_{q-1}}\otimes
dz^{k}\otimes dz^{j_{q+1}}\otimes dz^{j_s}
\end{align*}
In the expression above, we use the notation
$$R_{j\overline i}= \sum_p R_{p\overline p j\overline i}$$
for the coefficients of the Ricci tensor (here all the quantities 
are expressed with respect to some geodesic coordinates), as well as 
$I= (i_1,\ldots, i_r)$ and $J= (j_1,\ldots, j_s)$. 
\end{lemm}

\begin{proof} The arguments which will follow are very classical: we evaluate the quantity
$$\#\overline \partial^\star\dbar u- \overline \partial^\star\dbar (\# u)
$$
at some point $p$ of $X_0$; as we shall see, the preceding difference can be expressed
in terms of the operator $\cR$.

 To this end, we consider a geodesic coordinate system 
 $z= (z^1,\ldots, z^n)$ centered at $p$. If 
 $$w= \sum_{I, J, \alpha} w^I_{J, \overline\alpha}{\partial\over\partial z^{\otimes I}}
 \otimes {dz^{\otimes J}}\otimes dz^{\overline \alpha}$$
 is the local expression of a $(0, 1)$-form
 with values in $T^r_s(X_0)$, then we have
\begin{equation}
\label{24}
\dbar^\star \! w= -\sum_{I, J, \alpha} {\partial w^I_{J, \overline \alpha}\over \partial z^{\overline \alpha}}
 {\partial\over\partial z^{\otimes I}}
 \otimes {dz^{\otimes J}},
 \end{equation}
that is to say, the expression of the formal adjoint operator
$\dbar^\star$ at $p$ is precisely the same as the one corresponding to the 
flat metric in $\bC^n$. By using the notations in \eqref{23}, we have
\begin{equation}
\label{25}
\dbar u= \sum_{I, J}{\partial u^I_J\over \partial z^{\overline \alpha}}
{\partial\over \partial z^{\otimes I}}\otimes dz^{\otimes J}\otimes dz^{\overline \alpha}
\end{equation}
and by relation \eqref{24} we infer that we have
\begin{equation}
\label{26}
\dbar^\star \dbar u= -\sum_{I, J}{\partial^2u^I_J\over 
\partial z^{\alpha}\partial z^{\overline \alpha}}
{\partial\over \partial z^{\otimes I}}\otimes dz^{\otimes J}
\end{equation}
at $p$.

Next, given the preceding considerations/computations (cf. \eqref{20} and \eqref{21}), we have
$$\#\Big({\partial\over \partial z^{\otimes I}}\otimes dz^{\otimes J}\Big)=
\sum_{K, L}g^{K\overline J}_{L\overline I}{\partial\over \partial z^{\otimes K}}\otimes dz^{\otimes L}$$
for any pair of indexes $(I, J)$, hence we obtain
\begin{align*}
\#\dbar^\star \dbar u= &-\sum_{K, L}\sum_{I, J}\overline{{\partial^2u^I_J\over 
\partial z^{\alpha}\partial z^{\overline \alpha}}}g^{K\overline J}_{L\overline I}{\partial\over \partial z^{\otimes K}}\otimes dz^{\otimes L}=\\
= & -\sum_{I, J}\overline{{\partial^2u^I_J\over 
\partial z^{\alpha}\partial z^{\overline \alpha}}}{\partial\over \partial z^{\otimes J}}\otimes dz^{\otimes I}
\end{align*}
at the point $p$. \\

On the other hand, we have
$$\#u = \sum_{I, J}\sum_{K, L}\overline {u^I_J}g^{K\overline J}_{L\overline I}{\partial\over \partial z^{\otimes K}}\otimes dz^{\otimes L}$$
and it follows that
\begin{align*}
\dbar^\star \dbar (\#u)= & -\sum_{I, J}\sum_{K, L}\overline 
{{\partial^2u^I_J\over 
\partial z^{\alpha}\partial z^{\overline \alpha}}}g^{K\overline J}_{L\overline I}{\partial\over \partial z^{\otimes K}}\otimes dz^{\otimes L}- \\
- & \sum_{I, J}\sum_{K, L}\overline {u^I_J}
{\partial^2g^{K\overline J}_{L\overline I}\over 
\partial z^{\alpha}\partial z^{\overline \alpha}}{\partial\over \partial z^{\otimes K}}\otimes dz^{\otimes L};
\end{align*}
when evaluated at $p$, the first term of the right hand side
of the preceding equality is precisely $\#\dbar^\star \dbar u$; as for the second one,
it is equal to $\#\cR(u)$. 

In conclusion, we obtain
$$\overline \partial^\star\dbar (\# u)= \#\overline \partial^\star\dbar u+ \#\cR(u)
$$
at $p$. We consider the scalar product with $\# u$, and we use the fact that 
$\#$ is an isometry, together with the fact that $u$ has compact support (this enables us to
perform the integration by parts); the lemma is proved.
\end{proof}

\medskip

\noindent We start now the proof of Theorem \textbf{C}; we first discuss the point $(i)$, so let 
$$u\in H^0\big(X, T^r_s(X|D)\big)$$
be a holomorphic tensor field, where $r\ge 1+s$.
As mentioned in the introduction, we will 
have to use a cut-off function to the following purposes:

\begin{enumerate}  

\item We intend to show that $u=0$ by using the Bochner technique (Lemma \ref{lem33}), but the formula only applies to the compactly supported tensors.

\item We want the error term induced by the truncation procedure to converge to zero.\\

\end{enumerate}

\noindent Following \cite[Lemma 2.2]{bobot}, we define $\rho: X\to [-\infty, +\infty]$ by the formula
$$\rho(x):= \log\Big(\log {1\over \prod_j |s_j(x)|^2}\Big)$$
where we recall that $s_j$ is a section of the line bundle $L_j$, whose zero set is the 
hypersurface $Y_j$. \\

For each $\varepsilon> 0$, let $\displaystyle \Xi_{\varepsilon}: [0,+ \infty[\to [0, 1]$ be a smooth function which is 
equal to zero on the interval $\displaystyle [0, {1/\varepsilon}]$, and which is equal to 1
on the interval $\displaystyle [1+ {1/\varepsilon}, +\infty]$. One may for example define $\Xi_{\ep}(x)=\Xi_{1}(x-\frac{1}{\ep})$, so that
$$\sup_{\varepsilon> 0, t\in \bR_+}|\Xi_{\varepsilon}^\prime(t)|\le C < \infty,$$
and we define $\theta_\varepsilon: X\to [0, 1]$ by the expression
$$\theta_\varepsilon(x)= 1- \Xi_{\varepsilon}\big(\rho(x)\big).$$
\noindent We assume from the beginning that we have 
$$\prod_j|s_j|^2\le e^{-1}$$
at each point of $X$, and then it is clear that we have 
$$\theta_\varepsilon =1 \iff \prod_j|s_j|^2\ge e^{-e^{1/\varepsilon}}$$
and also
$$\theta_\varepsilon =0 \iff \prod_j|s_j|^2\le e^{-e^{1+ 1/\varepsilon}}.$$
We evaluate next the norm of the $(0, 1)$--form $\dbar \theta_\varepsilon$; we have
$$\dbar \theta_\varepsilon(x)=  \Xi_{\varepsilon}^\prime\big(\rho(x)\big)
{1\over \log {1\over \prod_j |s_j(x)|^2}}\sum_j{\langle s_j, D^\prime s_j\rangle
\over |s_j|^2}(x).$$
In the context of the part $(i)$ of Theorem \textbf{C} we have a K\"ahler-Einstein metric
$\omega_\infty$ on $X_0$, which satisfies inequality \eqref{12}; we obtain 
\begin{equation}
\label{27}
|\dbar \theta_\varepsilon|_{\omega_\infty}^2\le 
{C|\Xi_{\varepsilon}'(\rho)|^2\over \log^2 {1\over \prod_j |s_j|^2}}\sum_j{1\over |s_j|^{2\tau_j}}
\end{equation} 
at each point of $X_0$. Indeed, this is a consequence of the fact that the norm of the $(1, 1)$-form
$$\sqrt{-1}{\langle D^\prime s_j, D^\prime s_j\rangle
\over |s_j|^{2(1-\tau_j)}}$$
with respect to $\omega_\infty$ is bounded from above by a constant; again, we use the 
fact that the metric $\omega_\infty$
is assumed to have the properties at the end of section \ref{sec3}.\\

Let $\varepsilon> 0$ be a real number; we consider the tensor
$$u_\varepsilon:= \theta_\varepsilon u.$$
It has compact support, hence by Lemma \ref{lem41} we infer
\begin{equation}
\label{28}
\int_{X_0}|\overline\partial(\# u_\varepsilon)|^2dV_{\omega_\infty}= \int_{X_0}|\overline\partial u_\varepsilon |^2dV_{\omega_\infty}
+ \int_{X_0}\langle \cR(u_\varepsilon), u_\varepsilon\rangle dV_{\omega_\infty}
\end{equation}
\noindent The K\"ahler-Einstein condition reads as
$$R_{j\overline i}= -\delta_{ji},$$
and therefore the linear term $\displaystyle \langle \cR(u_\varepsilon), u_\varepsilon\rangle$ becomes simply $(s-r)|u_\varepsilon |^2$.

We show next that the term
$$\int_{X_0}|\overline\partial u_\varepsilon |^2dV_{\omega_\infty}$$
tends to zero as $\varepsilon\to 0$. Since $u$ is holomorphic, we have
$$\dbar u_\varepsilon= u\otimes \dbar \theta_\varepsilon;$$ 
we recall now that $u$ is a bounded tensor, in the sense given in section \ref{sec3}, so we have

\begin{equation}
\label{29}
|\dbar u_\varepsilon|^2\le C|\dbar \theta_\varepsilon|^2.
\end{equation}
By inequality \eqref{27} above we infer 
\begin{equation}
\label{30}
\int_{X_0}|\overline\partial u_\varepsilon |^2dV_{\omega_\infty}
\le C\int_{X_0}{|\Xi^\prime_\varepsilon (\rho)|^2\over \log^2 {1\over \prod_j |s_j|^2}}\sum_p{1\over |s_p|^{2\tau_p}}dV_{\omega_\infty}.
\end{equation}

\noindent As $\omega_\infty$ as cone singularities along $D$ (Theorem \textbf{A}), we have:
\begin{equation}
\label{31}
\int_{X_0}|\overline\partial u_\varepsilon |^2dV_{\omega_\infty}
\le C\int_{X_0}{|\Xi^\prime_\varepsilon (\rho)|^2\over \prod_j |s_j|^{2(1- \tau_j)}
\log^2 {1\over \prod_j |s_j|^2}}\sum_p{1\over |s_p|^{2\tau_p}}dV_{\omega}.
\end{equation}
for some constant $C> 0$ independent of $\varepsilon$; here we denote by 
$\omega$ a smooth hermitian metric on $X$. 
We remark that the support of the function $\displaystyle \Xi^\prime_\varepsilon (\rho)$ is contained in the set 
$$e^{-e^{1+1/\varepsilon}}\le \prod_j|s_j|^2\le e^{-e^{1/\varepsilon}}$$
so in particular we have
\begin{equation}
\label{32}
{|\Xi^\prime_\varepsilon (\rho)|^2\over 
\log^{1\over 2} {1\over \prod_j |s_j|^2}}\le Ce^{-{1\over 2\varepsilon}}.
\end{equation}
We also notice that for each index $p$ we have
$$\int_{X_0}{dV_{\omega}\over |s_p|^{2}\log^{3/2} {1\over \prod_j |s_j|^2}\prod_{j\neq p} 
|s_j|^{2(1- \tau_j)}}\le C \int_{X_0}{dV_{\omega}\over |s_p|^{2}\log^{3/2} {1\over |s_p|^2}\prod_{j\neq p} 
|s_j|^{2(1- \tau_j)}}
$$
and the last integral is convergent, given that the hypersurfaces 
$(Y_j)$ have strictly normal intersections. \\ 

We combine the inequalities \eqref{31}-\eqref{32}, and we get
\begin{equation}
\label{33}
\int_{X_0}|\overline\partial u_\varepsilon |^2dV_{\omega_\infty}
\le Ce^{-{1\over 2\varepsilon}}.
\end{equation}
As we can see, the relation \eqref{33} combined with the fact that
$$\langle \cR(u_\varepsilon), u_\varepsilon\rangle= (s- r)|u|^2$$ will give a contradiction if $u$ is not identically zero (we recall that by hypothesis 
we have $r\ge s+ 1$).

\medskip
\noindent The point $(ii)$ of Theorem \textbf{C} is treated in a similar manner, except that 
$X_0$ may not admit a K\"ahler-Einstein metric. Nevertheless 
for any compact set $K\subset X_0$, the Ricci curvature of $\omega_\infty$ is greater than a small, positive multiple of 
$\omega_\infty$ on $K$ (cf. equality \eqref{10}), hence we get the result.

\medskip
\noindent We explain next our arguments for the point $(iii)$ of Theorem \textbf{C}; we only discuss the case
$$c_1(K_X+ D)\le 0.$$
The term 
\begin{equation}
\label{34}
\int_{X_0}\langle \cR(u_\varepsilon), u_\varepsilon\rangle dV_{\omega_\infty}
\end{equation}
is negative or zero, for any $\varepsilon> 0$ (here $u$ is a holomorphic 
tensor of type $(r, 0)$). Inequality \eqref{33} still holds;
in conclusion, we have 
\begin{equation}
\label{35}
\int_{X_0}|\overline\partial(\# u_\varepsilon)|^2dV_{\omega_\infty}\to 0,
\end{equation}
as $\varepsilon\to 0$. We remark next that we have
\begin{equation}
\label{36}
\int_{X_0}|\overline\partial \theta_\varepsilon\otimes \# u|^2dV_{\omega_\infty}\to 0,
\end{equation}
since the tensor $\# u$ is bounded 
(as we see thanks to Lemma \ref{lem33}), so the inequality \eqref{33} gives
\begin{equation}
\label{37}
\int_{X_0}\theta_\varepsilon|\overline\partial (\# u)|^2dV_{\omega_\infty}\to 0
\end{equation}
as $\varepsilon\to 0$. But this means that the tensor $\# u$
is holomorphic on $X_0$; in other words, we have $D^\prime u= 0$
at each point of $X_0$. By Lemma \ref{lem32}, we equally infer that we have
$$\# u\in H^0\big(X, T_r(X|D)\big).$$ 
\noindent In conclusion, the function
$$x\to |u(x)|_{\om_{\infty}}^2$$
is constant, henceforth $u$ is parallel with respect to
$\omega_\infty$. The proof of Theorem \textbf{C} is finished.

\begin{rema}If we are trying to implement the cut-off 
procedure described here by considering the "standard" function
$$\theta_\varepsilon:= \theta\Big({1\over \varepsilon^2}{\prod_j |s_j|^2}\Big)$$ 
then the boundary term 
\begin{equation}
\label{38}
\int_{X_0}|\overline\partial \theta_\varepsilon\otimes u|^2dV_{\omega_\infty}
\end{equation}
does not seem to tend to zero, as $\varepsilon\to 0$. Indeed, even if the support of 
$D= (1- \tau) Y$ is irreducible, we have
$$|\dbar \theta_\varepsilon|^2_{\omega_\infty} \le \theta^\prime(|s|^2/\varepsilon^2)^2{|s|^{4-2\tau}\over \varepsilon^4}$$
and the above quantity integrated against the measure
$${1\over |s|^{2(1- \tau)}}dV_\omega$$
is bounded, but it does not converges to zero as $\varepsilon\to 0$.
\end{rema}

\section{Further comments}
A question which seems to be very interesting to us \---at least from the point of view of the analysis which is required\--- would be to formulate and prove an 
analog result in the log-canonical case, i.e. prove Theorem \textbf{A} for a divisor $D$ having the following shape
$$D= \sum_{j\in J_1}Y_j+ \sum_{j\in J_2}(1-\tau_j)Y_j$$
where as always the hypersurfaces $(Y_j)$ are smooth, and have strictly normal intersections.\\

Another problem would be to consider the case where the representative of the class $c_1(K_X+ D)$ is singular. For example, we assume that $(X, D)$ is klt, and that $K_X+ D$ is
big. Then the analogue of the K\"ahler-Einstein metric in this context
is constructed e.g. in \cite{BEGZ}, but it is not clear how this object 
can be used in order to obtain some geometric consequences as in Theorem \textbf{A}.

\bibliographystyle{smfalpha}
\bibliography{Biblio.bib}
\vspace{3mm}

\end{document}